\documentclass[reqno]{amsart}
\usepackage{amssymb,amsmath}
\usepackage{amsthm}
\usepackage{stmaryrd}
\usepackage{mathabx}
\usepackage{color,graphicx}
\usepackage{xcolor}
\usepackage{hyperref}
\usepackage{verbatim}
\usepackage{dsfont}

\usepackage{stmaryrd}

\usepackage{bbm}

\DeclareMathOperator{\dist}{\mathrm{dist}}
\DeclareMathOperator\supp{supp}

\newcommand{\vertiii}[1]{{\left\vert\kern-0.25ex\left\vert\kern-0.25ex\left\vert #1 
    \right\vert\kern-0.25ex\right\vert\kern-0.25ex\right\vert}}

\numberwithin{equation}{section}

\newtheorem{theorem}{Theorem}[section]
\newtheorem{proposition}[theorem]{Proposition}
\newtheorem{remark}[theorem]{Remark}
\newtheorem{lemma}[theorem]{Lemma}
\newtheorem{corollary}[theorem]{Corollary}

\newtheorem{lthmx}{{\bf Lemma}}

\begin{document}

\title[On Kato's smoothing effects for KdV and Benjamin type equations]{On Kato's smoothing effects for KdV and Benjamin type equations}



\author[C. Garz\'on]{Carlos Garz\'on}
\address{Department of Mathematics, Virginia Polytechnic Institute and State University,  225 Stanger Street, Blacksburg, VA 24061-1026, USA}
\email{cgarzongu@vt.edu}

\author[O. Ria\~no]{Oscar Ria\~no}
\address{Departamento de Matem\'aticas, Universidad Nacional de Colombia, Ca\-rre\-ra 30 No. 45-03, 111321, Edificio Yu Takeuchi 404-209, Bogot\'a D.C., Colombia}
\email{ogrianoc@unal.edu.co}




\begin{abstract}
We analyze how the interaction between local and nonlocal dispersions, combined with different types of nonlinearities, influences the smoothing effects of solutions. To achieve this goal, we consider a model that generalizes the KdV and Benjamin equations and demonstrate that its solutions exhibit Kato's smoothing effect and satisfy the propagation of regularity principle. As a result, we confirm that the higher-order dispersive term determines the local gain of fractional regularity of solutions. Our results are general; they not only recover known results for the KdV and Benjamin equations, but also provide new insights for a broader family of models of physical and mathematical interest with polynomial dispersions of arbitrary order.
\end{abstract}
 
\maketitle

\section{Introduction}\label{introduction}

This paper is concerned with the initial value problem (IVP) associated to the following Benjamin-type equations
\begin{equation}\label{Benjamineq}
\left\{\begin{aligned}
&\partial_t u+\gamma\mathcal{H}\partial_x^2 u+(-1)^{N+1} \partial_x^{2N+1}u + P(D) u+ \sum_{k=1}^{M}b_k u^k\partial_x u = 0, \quad x\in \mathbb{R},\,  t\in \mathbb{R},\\
&u(x,0)=u_0(x), 
\end{aligned}\right.    
\end{equation}
where the unknown $u(x,t)$ is a real-valued function, $\mathcal{H}$ denotes the Hilbert transform, $N$, $M$ are positive integer numbers, $\gamma \in \mathbb{R}$, the constants $b_k\in \mathbb{R}$, for each $k\in\{1,\dots,M\}$ with $b_M\neq 0$, and the differential operator $P(D)$ is composed of linear combinations of odd order derivative operators, and is defined as follows
\begin{itemize}
    \item If $N=1$, then $P(D)=0$.
    \item If $N\geq 2$, then $P(D)=\sum\limits_{k=1}^{N-1} a_k \partial_x^{2k+1}=i\sum\limits_{k=1}^{N-1} (-1)^k  a_k D^{2k+1}$ with $D=-i\partial_x$, where the $a_k$'s are real numbers.
\end{itemize}
This work aims to analyze how different dispersive terms and nonlinearities affect the following phenomena: \emph{Kato's smoothing effect} and the \emph{propagation of the regularity principle}. To this end, we introduce the model in \eqref{Benjamineq}, which encompasses various well-known nonlinear dispersive equations, and provides a rigorous framework for studying different dispersive and nonlinear behaviors. See Subsection~\ref{Subexampcontr} below for concrete examples covered by \eqref{Benjamineq}. As a consequence of our result, we deduce that the higher-order dispersive term in the equation \eqref{Benjamineq} mainly governs the aforementioned smoothing effects.

Let us briefly describe Kato's work on smoothing effects in \cite{Kato1983}. Setting $\gamma=0$, $N=1$, and $M=1$ in the equation in \eqref{Benjamineq}, one recognizes the widely studied Korteweg-de Vries (KdV) equation
  \begin{equation}\label{KdVeq}
      \begin{aligned}
       \partial_t u+\partial_x^{3}u + b\, u \partial_x u=0,  
      \end{aligned}  
    \end{equation}
where $b\in \mathbb{R}$, with $b\neq 0$. The KdV equation was first proposed to model the unidirectional propagation of nonlinear dispersive long waves, and it has been applied in various mathematical and physical contexts. Focusing on smoothing effects, Kato in \cite{Kato1983} deduced that if $u\in C([0,T];H^{s}(\mathbb{R}))$ solves \eqref{KdVeq} with regularity $s>3/2$, then 
\begin{equation*}
    \partial_x J^su\in L^{2}([0,T];L^2(-R,R))
\end{equation*}
for any $R>0$. In other words, the solution to the initial value problem is locally one derivative smoother than the initial data. This is an interesting result as the dispersive character of the equation does not immediately indicate a gain of regularity. In this regard, it is worth pointing out that smoothing effects are an intrinsic property of linear dispersive equations, for which it has been proved that the higher-order dispersion dominates the gain of regularity.  We refer to the works of \cite{ConstantinSaut1988,KenigPonceVega1991I,Sjolin1987,Vega1988}. Additionally, we remark that Kato's smoothing effect has been deduced for fractional generalizations of the KdV equation (e.g., the generalized Benjamin-Ono equation), see \cite{GinibreVelo1991,Ponce1990}.  

In this work, we propose to study whether, for an arbitrary combination of dispersion parameters and nonlinearities as in \eqref{Benjamineq}, the resulting solutions of the equation satisfy a smoothing effect similar to that of Kato for the KdV.

On the other hand, Kato's work has served as a basis for analyzing other types of regularizing effects. Among these, we focus on the work of Isaza, Linares, and Ponce \cite{IsazaLinaresPonce2015}, who established what nowadays is known as the \emph{propagation of regularity principle} for solutions of the KdV equation. They proved that extra regularity in the initial data localized on the right-hand side of the real line travels to the left with infinite speed as time evolves. That is, if $u_0\in H^s(\mathbb{R})$ (for appropriated $s>0$, e.g., $s>3/4$, see \cite{IsazaLinaresPonce2015}) and it has the extra regularity condition $\partial_x^l u_0\in L^2((x_0,\infty))$, for some $l\geq 1$ and $x_0\in \mathbb{R}$, then the solution $u\in C([0,T];H^s(\mathbb{R}))$ of \eqref{KdVeq} emanating from the initial condition $u_0$ satisfies, for any $v>0$ and $\varepsilon>0$, that
\begin{equation*}
  \sup_{0\leq t \leq T}\int_{x_0+\varepsilon-vt}^{\infty} (\partial_x^l u)^2(x,t)\, dx<\infty.
\end{equation*}
We remark that the study of the propagation of regularity principle has been extended to fractional differential operators (see \cite{KenigLinaresPonceVega2018}), and it has been investigated for various dispersive equations; for a recent survey, see \cite{LinaresPonce2023}. We also invite the reader to consult \cite{FreireMendezRiano2022, IsazaLinaresPonce2015,IsazaLinaresPonceBO2016,BolingGuoquan2018, KenigLinaresPonceVega2018} and references therein. Consequently, the second main result of this paper shows that the propagation of regularity principle holds not only for integer derivatives, but also for fractional differential operators for solutions of \eqref{Benjamineq}.

Let us mention some invariants of the equations \eqref{Benjamineq}. Solutions to \eqref{Benjamineq} formally conserve \emph{the mass}
\begin{equation}\label{Mass}
 M[u(t)]:=\int_{\mathbb{R}} u^2(x,t)\, dx,   
\end{equation}
the \emph{energy or Hamiltonian} 
\begin{equation}\label{energy}
\begin{aligned}
E[u(t)]:=&\frac{1}{2}\int_{\mathbb{R}} (\partial_x^N u)^2(x,t)\, dx-\frac{\gamma}{2}\int_{\mathbb{R}} (|D|^{1/2}u)^2(x,t)\, dx-\frac{1}{2}\sum_{k=1}^{N-1} (-1)^k a_k\int_{\mathbb{R}} (\partial_x^k u)^2(x,t)\, dx\\
&-\sum_{k=1}^{M} \frac{b_k}{(k+1)(k+2)} \int_{\mathbb{R}}u^{k+2}(x,t)\, dx,      
\end{aligned}
\end{equation}
and the \emph{$L^1$-type} conservation
\begin{equation*}
  I[u(t)]:=\int_{\mathbb{R}} u(x,t)\, dx.  
\end{equation*}

We note that in some cases, if there are more than two dispersive terms in the equation in \eqref{Benjamineq} (i.e., if $N=1$, $\gamma\neq 0$, or when $N\geq 2$, either $\gamma\neq 0$ or $a_k\neq 0$ for some $k\in \{1,\dots N-1\}$), then the equation does not have scale invariant.


\subsection{Main results}

Next, we present the main results obtained in this document. But first, since our results depend on energy estimates and pseudo-differential calculus techniques, we require a local well-posedness (LWP) theory which ensures that the norm $\|\partial_x u\|_{L^q_T L^{\infty}_x}$, $q>1$ is finite for solutions of \eqref{Benjamineq}. In this direction, we present the following lemma.

\begin{lthmx}\label{LWPN12}
Let $N, M \in \mathbb{Z}^+$, and let $\gamma \in \mathbb{R}$. If $N \geq 2$, consider $ a_1, \dots, a_{N-1} \in \mathbb{R} $. In addition, let $ b_1, \dots, b_M \in \mathbb{R} $ with $ b_M \neq 0 $. Let $s>\frac{2N+1}{4}$, then the Cauchy problem \eqref{Benjamineq} is LWP in $H^s(\mathbb{R})$. More precisely, for any $u_{0}\in H^s(\mathbb{R})$ there exist a time $T>0$ and a unique solution $u(t)$ of the Cauchy problem \eqref{Benjamineq} satisfying
\begin{equation}\label{solclass1}
   u\in C([0,T]; H^s(\mathbb{R})) 
\end{equation}
and
\begin{equation}\label{solclass2}
  \partial_x J^{r}u\in L^4([0,T];L^{\infty}(\mathbb{R})),  
\end{equation}
for all $0\leq r\leq s-\frac{(5-2N)}{4}$. Moreover,
\begin{equation}\label{solclass3}
  \||D|^s\partial_x u \|_{L^{\infty}_x L^2_T}<\infty
\end{equation}
and 
\begin{equation}\label{solclass4}
 \|u\|_{L^2_xL^{\infty}_T}<\infty.   
\end{equation}
Additionally, the map data-to-solution, $u_0\rightarrow u(x,t)$ is locally continuous from $H^s(\mathbb{R})$ into the class defined by \eqref{solclass1}-\eqref{solclass4}.
\end{lthmx}

The proof of Lemma \ref{LWPN12} relies on a contraction mapping argument applied to the integral formulation of \eqref{Benjamineq}. This approach incorporates Strichartz estimates, smoothing effects, and maximal function bounds. The overall strategy follows techniques developed in \cite{KenigPonceVega1991,Laurey1997}. We discuss the proof of Lemma \ref{LWPN12} in the Appendix \ref{LWPresults} below.

The results of Lemma \ref{LWPN12} for $N\geq 3$ do not improve the regularity provided by standard results of LWP, which do not take into account dispersive effects. These techniques typically work for regularities $s>3/2$, and the resulting solutions satisfy (via Sobolev embedding) $\|\partial_x u\|_{L^q_TL^{\infty}_x}\lesssim T^{\frac{1}{q}}\|u\|_{L^{\infty}_TH^s}<\infty$, for any $q\geq 1$. For these reasons, by applying tools such as Kato's abstract theory \cite{Kato1975}, we obtain the following result.

\begin{lthmx}\label{StandardLWP}
Let $N, M \in \mathbb{Z}^+$, and let $\gamma \in \mathbb{R}$. If $N \geq 2$, consider $ a_1, \dots, a_{N-1} \in \mathbb{R} $. In addition, let $ b_1, \dots, b_M \in \mathbb{R} $ with $ b_M \neq 0 $. Let $s>\frac{3}{2}$. Then the Cauchy problem \eqref{Benjamineq} is LWP in $H^s(\mathbb{R})$.
\end{lthmx}

We are now in a position to state the main results in the manuscript. Firstly, we deduce that solutions of \eqref{Benjamineq} satisfy Kato's smoothing effect. 
 
\begin{theorem} \label{SmoothingEffectThrm}
Let $N, M \in \mathbb{Z}^+$, and let $\gamma \in \mathbb{R}$. If $N \geq 2 $, consider $ a_1, \dots, a_{N-1} \in \mathbb{R} $. In addition, let $ b_1, \dots, b_M \in \mathbb{R} $ with $ b_M \neq 0 $. Consider $s>\min\left\{\frac{3}{2} , \frac{2N+1}{4}\right\}$, and let $u_0 \in H^s(\mathbb{R})$. Then the corresponding solution $u \in C([0,T];H^s(\mathbb{R}))$ of the IVP \eqref{Benjamineq} with initial data $u_0$ determined by Lemmas \ref{LWPN12} and \ref{StandardLWP} satisfies that, for each $r\in [0,s+N]$, 
\begin{equation}\label{Katosmoothingeq}
   \int_0^T \int_{-R}^{R} (A^r u)^2 \, dx \, dt  \leq c(\|u_0\|_{H^s};T;R;r) 
\end{equation}
for every positive real number $R$, where $A^r$ is any of the operators $J^r$, $|D|^r$, and $\partial_x^{\lfloor r \rfloor} |D|^{\{r\}}$. Here, $\lfloor r \rfloor$ and $\{r\}$ stand for the integral and fractional parts of the real number $r$, respectively. 
\end{theorem}

Theorem \ref{SmoothingEffectThrm} says that the local gain of regularity of \eqref{Benjamineq} is completely determined by the higher-order dispersive term of the equation. Such a result is expected to hold from the work of Constantin and Saut \cite{ConstantinSaut1988} on smoothing effects for general linear dispersive equations. Consequently, we observe that the nonlinear model exhibits a smoothing effect similar to its linear counterpart. 

The strategy for proving Theorem \ref{SmoothingEffectThrm} is based on pseudo-differential calculus, commutator estimates (expansions of commutators and operators such as $|D|^s-J^s$), which allow us to deal with different fractional differential operators. For similar techniques, see \cite{KenigLinaresPonceVega2018,Mendez2024,FreireMendezRiano2022,GinibreVelo1991}. However, we emphasize that the main difference in our work lies in applying Kato's approach in \cite{Kato1983} with multiple dispersive terms, while dealing with fractional regularities. It is worth mentioning that to mitigate the influence of different dispersion, we begin by working with energy estimates with the operator $J^{s-N+1}$, so that the higher-order dispersive term acts as a regularizer at the first derivative level. This allows us to execute a similar argument with regularity $J^{s-N+2}$, thus obtaining local gain of a second derivative, and so on. 

Using our local smoothing effects in Theorem \ref{SmoothingEffectThrm}, and pseudo-differential calculus, we deduce that solutions of \eqref{Benjamineq} satisfy the propagation of regularity principle.

\begin{theorem}\label{ThrmPropagation}
Let $N, M \in \mathbb{Z}^+$, and let $\gamma \in \mathbb{R}$. If $N \geq 2$, consider $ a_1, \dots, a_{N-1} \in \mathbb{R} $. In addition, let $ b_1, \dots, b_M \in \mathbb{R} $ with $ b_M \neq 0 $. Consider $s>\min\left\{\frac{3}{2} , \frac{2N+1}{4}\right\}$, and let $u_0 \in H^s(\mathbb{R})$. If for some $x_0\in\mathbb{R}$ and for some $m \in (s,\infty)$ 
    \begin{equation} \label{hypoth}
        \|J^m u_0\|_{L^2((x_0.\infty))}^2 = \int_{x_0}^{\infty} |J^m u_0(x)|^2 \, dx < \infty,
    \end{equation} 
    then the solution $u=u(x,t)$ of the IVP \eqref{Benjamineq} provided by Lemmas \ref{LWPN12} and \ref{StandardLWP} satisfies that for any $v,\varepsilon>0$ and $r\in (s,m]
    $
    \begin{equation} \label{propagation}
        \sup_{0\leq t \leq T} \int_{x_0+\varepsilon-vt}^{\infty}  (J^r u)^2(x,t) \, dx \leq C<\infty,
    \end{equation}
    with $C = C\left(m;\|u_0\|_{H^s};\|J^r  u_0\|_{L^2((x_0,\infty))};\varepsilon;v;T\right)>0$.

    Moreover, for any $v\geq 0$, $\varepsilon>0$, and $R>0$ 
    \begin{equation} \label{smoothing}
        \int_0^T \int_{x_0+\varepsilon-vt}^{x_0+R-vt}  (J^{m+N}u)^2(x,t) \, dx \, dt \leq C<\infty
    \end{equation}
    with $C = C\left(m;\|u_0\|_{H^s};\|J^m u_0\|_{L^2((x_0,\infty))};\varepsilon;v;R;T\right)>0$.
\end{theorem}
The deduction of this theorem follows from energy estimates with appropriate weighted functions in the spirit of the works of \cite{IsazaLinaresPonce2015,KenigLinaresPonceVega2018}, see also \cite{FreireMendezRiano2022,Mendez2024}.  Our arguments in the proof of Theorem \ref{ThrmPropagation} allow us to control dispersive and nonlinear terms in a general way, and again establish that the higher-order dispersion essentially governs the propagation of regularity.

On the other hand, arguing as in the proof of \cite[Corollary 2.3]{LinaresPonce2023}, our deduction of Theorem \ref{ThrmPropagation} yields a more precise version of \eqref{propagation}.

\begin{corollary}\label{corollarypropdecay}
Under the assumptions of Theorem \ref{ThrmPropagation} with $m\in (s,\infty)$, and $r\in (s,m]$, for any $\delta>0$, $t\in (0,T)$ it follows
\begin{equation} \label{propagationexpand}
        \int_{-\infty}^{\infty} \frac{1}{\langle x_{-} \rangle^{\lfloor r-s\rfloor+1+\delta}} (J^r u)^2(x,t) \, dx \leq\frac{C}{t},
    \end{equation}
where $C=C(\|u_0\|_{H^s},\|J^m u_0\|_{L^2((x_0,\infty))},x_0;\delta)$ and $x_{-}=\max\{0;-x\}$.    
\end{corollary}

\begin{remark}

(a) The LWP results in Lemmas \ref{LWPN12} and \ref{StandardLWP} are far from the best. We do not focus on the theory of minimal regularity; instead, we seek to study smoothing effects. However, we remark that Theorems \ref{SmoothingEffectThrm} and \ref{ThrmPropagation} extend to LWP theories for regularities $H^s(\mathbb{R})$, $s>\frac{1}{2}$ for which $\|\partial_x u\|_{L^{q}_TL^{\infty}_x}<\infty$ for some $q\geq 2$. Additionally, our LWP results establish that for some nonlinearities, \eqref{Benjamineq} is globally well-posed (GWP) in the energy space $H^N(\mathbb{R})$. More precisely, using interpolation and  Gagliardo-Nirenberg inequality, together with the mass \eqref{Mass} and energy \eqref{energy} conservations, one deduces, for example:
 \begin{itemize}
\item[(i)] If $M<4N$, then the Cauchy problem \eqref{Benjamineq} is GWP in $H^N(\mathbb{R})$.
\item[(ii)] If $M=4N$ and $\|u_0\|_{L^2}$ is small enough, then the solution $u\in C([0,T);H^N(\mathbb{R}))$ of \eqref{Benjamineq} with initial data $u_0$ provided by Lemmas \ref{LWPN12} or \ref{StandardLWP} exists for arbitrary time $T>0$. 
\end{itemize}

(b) Another question somehow connected with this work is the study of propagation of regularity with localized extra polynomial or exponential weights, i.e., study how the solution flow of \eqref{Benjamineq} propagates the extra conditions on the initial data $\langle x\rangle^r u_0\in  L^2((x_0,\infty))$, or $e^{x} u_0\in  L^2((x_0,\infty))$. In the non-local case, $\gamma\neq 0$ in \eqref{Benjamineq}, such a propagation it is not expected to hold for polynomial weights of arbitrary order (much less exponential weights), we refer to \cite{IsazaLinaresPonceBO2016}. For the local case, $\gamma=0$ in \eqref{Benjamineq}, following results as in \cite{MendezRiano2025}, it might be possible to establish the propagation of polynomial and exponential weights. However, developing this analysis would require additional techniques and would significantly divert the focus of the present paper. For this reason, we will not address this issue here.
\end{remark}

\begin{remark} Replacing $\mathcal{H}\partial_x^2$ by the general operator $\partial_x |D|^{\beta}$ with $\beta\in (0,2)$, we obtain the following version of the equation \eqref{Benjamineq}
\begin{equation}\label{Dispersiveeq}
\begin{aligned}
\partial_t u+\gamma\partial_x |D|^{\beta} u+(-1)^{N+1} \partial_x^{2N+1}u + P(D) u+ \sum_{k=1}^{M}b_k u^k\partial_x u = 0, \quad x\in \mathbb{R}, \, t\in \mathbb{R}. 
\end{aligned}   
\end{equation} 
With the techniques developed in this work, Theorems \ref{SmoothingEffectThrm} and \ref{ThrmPropagation} extend to this family of equations. Notice that using standard techniques such as Kato's abstract theory, it is possible to show LWP for the Cauchy problem associated to \eqref{Dispersiveeq} in $H^{s}(\mathbb{R})$, with $s>3/2$.  Summarizing:

\begin{corollary}
Let $N, M \in \mathbb{Z}^+$, $\beta \in (0,2)$, and $\gamma\in \mathbb{R}$. If $ N \geq 2 $, consider $ a_1, \dots, a_{N-1} \in \mathbb{R} $. In addition, let $ b_1, \dots, b_M \in \mathbb{R} $ with $ b_M \neq 0 $.  For $s>\frac{3}{2}$ and initial condition $u_0\in H^s(\mathbb{R})$, the conclusions of Theorems \ref{SmoothingEffectThrm} and \ref{ThrmPropagation} also hold for solutions of the Cauchy problem associated to \eqref{Dispersiveeq}. In particular, under each corresponding assumption, the estimates \eqref{Katosmoothingeq}, \eqref{propagation}, and \eqref{smoothing} remain valid for \eqref{Dispersiveeq}.   
\end{corollary}
\end{remark}

\subsection{Examples and contributions}\label{Subexampcontr}

Let us exhibit some examples within the class in \eqref{Benjamineq}, where our results in Theorems \ref{SmoothingEffectThrm} and \ref{ThrmPropagation} are applicable.

$ \bullet$ Setting $\gamma=0$, and $P(D)=0$ (i.e., $a_k=0$ for all $k=1,\dots,N-1$, when $N\geq 2)$ in the equation in \eqref{Benjamineq}, one recognizes the $N$-th order Korteweg-de Vries (KdV) equation with combined nonlinearities
    \begin{equation}\label{NkdV}
      \begin{aligned}
       \partial_t u+(-1)^{N+1}\partial_x^{2N+1}u + \sum_{k=1}^{M}b_k u^k \partial_x u=0.  
      \end{aligned}  
    \end{equation}
In particular, when $N=1$, one obtains KdV \eqref{KdVeq} with combined nonlinearities. Concerning well-posedness, see \cite{KenigPonceVega1991,KillipVisan2019}. As a consequence of Theorems \ref{SmoothingEffectThrm} and \ref{ThrmPropagation}, we recover the smoothing effects and propagation of regularity principle obtained in \cite{IsazaLinaresPonce2015,KenigLinaresPonceVega2018} for solutions of KdV. 

Our result in Theorem \ref{SmoothingEffectThrm} is consistent with that in \cite{GinibreVelo1991}, although it provides an alternative proof. To the best of our knowledge, Theorem \ref{ThrmPropagation} is among the first to show the fractional propagation of regularity principle for solutions of \eqref{NkdV} with $N\geq 2$, and $M\geq 1$, $b_M\neq 0$.

$\bullet$ When $\gamma\neq 0$ and $N=1$, the equation in \eqref{Benjamineq} reduces to the Benjamin equation 
     \begin{equation}\label{BenjamineqOrig}
      \begin{aligned}
       \partial_t u+\gamma \mathcal{H}\partial_x^2 u+\partial_x^{3}u + \sum_{k=1}^{M}b_k u^k \partial_x u=0.  
      \end{aligned}  
    \end{equation}
For the case $M=1$, Benjamin made one of the first mentions of this model in \cite{Benjamin1996}, where it is applied to study waves in a two-fluid system subject to surface tension. Referring to well-posedness, see \cite{Linares1999,HuoGuo2005,ChenGuoXiao2011}. For a recent review of the Benjamin equation, see \cite{KleinLinaresPilodSaut2025} and references therein. 

Regarding smoothing effects, Guo and Qin \cite{BolingGuoquan2018} proved the propagation of regularity principle for \eqref{BenjamineqOrig} when $M=1$, and the initial data lies in $H^{l}((x_0,\infty))$, with $l\in\mathbb{Z}^{+}$ and $x_0\in\mathbb{R}$. Therefore, an extension of this result to Sobolev spaces $H^{s}((x_0,\infty))$ with non-integer $s$ and with a more general nonlinear term is provided by Theorem \ref{ThrmPropagation}.
\\ \\
Extending \eqref{BenjamineqOrig}, Chen and Bona \cite{ChenBona1998} (see also \cite{LinaresScialom2005}) introduced the model
    \begin{equation}\label{generalBenjamineqOrig}
      \begin{aligned}
       \partial_t u+\gamma\mathcal{H}\partial_x^2 u+(-1)^{N+1}\partial^{2N+1}_xu+\sum_{k=1}^{M}b_k u^k \partial_x u=0,  
      \end{aligned}  
    \end{equation}
which corresponds to the equation in \eqref{Benjamineq} with $P(D)=0$. In this paper, we obtain new fractional propagation of regularity results for \eqref{generalBenjamineqOrig}.

$ \bullet$ Setting $\gamma= 0$ and $N=2$, we get the Kawahara equation with combined nonlinearities
     \begin{equation*}
      \begin{aligned}
       \partial_t u-\partial_x^{5}u+a_1 \partial_x^3 u+ \sum_{k=1}^{M}b_k u^k \partial_x u=0.  
      \end{aligned}  
    \end{equation*}
This equation has been applied in modeling gravity-capillary waves. Regarding well-posedness for other nonlinearities, we refer to \cite{JiaHuo2009,Kato2011,KenigPilod2015,BrigmannKillipVisan2021,KlausKochBaoping2023} and references therein. Concerning smoothing effects, in \cite{GuoZhang2023,SegataSmith2017}, it is shown that a class of fifth-order KdV equations with more general nonlinear terms satisfy the propagation of regularity principle for the local derivative operators $\partial_x^{l}u_0\in L^2((x_0,\infty))$. Theorem \ref{ThrmPropagation} is different from the above results as the nonlinear terms considered are not the same and the pseudo-differential calculus allows us to prove the propagation of regularity principle for fractional regularity, i.e., in the space $H^s((x_0,\infty))$, $s>0$. Thus, our results provide new propagation of regularity conclusions in the non-integer case, and we believe that our techniques could be adapted to the model studied in \cite{GuoZhang2023,SegataSmith2017}.

$ \bullet$ Choosing  $\gamma\neq 0$, $N=2$, and $a_1\neq 0$, we obtain the fifth-order Benjamin-equation
     \begin{equation*}
      \begin{aligned}
       \partial_t u+\gamma \mathcal{H}\partial_x^2 u-\partial_x^{5}u+a_1\partial_x^3 u+ \sum_{k=1}^{M}b_k u^k \partial_x u=0.  
      \end{aligned}  
    \end{equation*}
This model appears in the study of different limits in the context of electric fields, e.g., \cite{GleesonHammertonPapaVanden2007}. For well-posedness results when $ M=1$, see \cite{CorreaMsC2023}. To the best of our knowledge, the smoothing effects for the solutions of this model have not been addressed before. Thus, Theorems \ref{SmoothingEffectThrm} and \ref{ThrmPropagation} are the first advancements in this direction. 

$\bullet$ Setting $\gamma=0$, when $N=3$ we obtain the seventh-order KdV
\begin{equation}\label{7th}
\partial_t u+\partial_x^7 u+a_2\partial_x^5 u +a_1\partial_x^3 u + \sum_{k=1}^{M}b_k u^k \partial_x u=0,   
\end{equation}
and with $N=4$, we have the ninth-order KdV
\begin{equation}\label{9th}
\partial_t u-\partial_x^9 u+a_3\partial_x^7 u+a_2\partial_x^5 u +a_1\partial_x^3 u + \sum_{k=1}^{M}b_k u^k \partial_x u=0.   
\end{equation}
These equations have appeared when studying the stability of perturbations of solitons, see \cite[Chapter 9]{Handbook2008} and \cite{Ma1993}. Additionally, the Cauchy problem associated to the family \eqref{7th} with $M=1$ has been studied in \cite{Wang2014}. {
Theorems \ref{SmoothingEffectThrm} and \ref{ThrmPropagation} seem to be the first results establishing fractional smoothing effects for solutions of \eqref{7th} and \eqref{9th}.


\subsection*{Organization of the document} 

This paper is organized as follows. In Section \ref{Secprelimin}, we present the necessary preliminaries to establish our main results, including some results from the theory of pseudo-differential operators as well as some commutator estimates. In Section \ref{Katosection}, we prove our first main result, Theorem \ref{SmoothingEffectThrm}, which concerns the deduction of Kato's smoothing effect for solutions of \eqref{Benjamineq}. In Section \ref{propSection}, we prove that solutions of \eqref{Benjamineq} satisfy the propagation of regularity principle, i.e., we establish Theorem \ref{ThrmPropagation}. Finally, in Appendix \ref{LWPresults}, we provide some key estimates in the deduction of LWP detailed in  Lemma \ref{LWPN12}.


\subsection*{Notation} 

\begin{itemize}
\item Given two nonnegative constants $a$ and $b$, $a\lesssim b$ means that there exists a constant $c>0$ such that $a\leq c b$. $a\gtrsim b$, whenever $b\lesssim a$. We say that $a\sim b$, when $a\lesssim b$ and $b\lesssim a$. To emphasize the dependence of the implicit constants, $a\lesssim_k b$ means that the implicit constant depends on $k$.

\item If $A$ and $B$ are two operators, $[A,B]=AB-BA$ stands for the commutator of $A$ and $B$. 

\item Given $x \in \mathbb{R}^d$, we use the Japanese bracket notation $\langle x \rangle = (1 + |x|^2)^{1/2}$.

\item $\mathcal{C}_0^{\infty}(\mathbb{R}^d)$ denotes the set of smooth functions defined on $\mathbb{R}^d$ with compact support. $S(\mathbb{R}^d)$ denotes the Schwartz  space of all $\mathcal{C}^{\infty}(\mathbb{R}^d)$ functions whose derivatives are rapidly decreasing. 

\item Given a function $f$, we denote its Fourier transform by $\widehat{f}$ or $\mathcal{F}f$, and we denote its inverse Fourier transform by $\widecheck{f}$ or $\mathcal{F}^{-1}f$. 

\item For each $s\in\mathbb{R}$, the operators $|D|^s = (-\Delta)^{s/2}$ and $J^s = (1 - \Delta)^{s/2}$ denote the Riesz and Bessel potentials of order $-s$, respectively. That is, given $f \in S(\mathbb{R}^d)$, these operators are defined by $|D|^s f = (|\xi|^s \widehat{f})^{\widecheck{}}$ (if $s<0$, definition is in the distributional sense) and $J^s f = (\langle \xi \rangle^s \widehat{f})^{\widecheck{}}$.

\item $H^s(\mathbb{R}^d)$ denotes the space of tempered distributions $f$ for which $J^s f\in L^2(\mathbb{R}^d)$. We denote by $H^{\infty}(\mathbb{R}^d)=\bigcap_{s>0} H^{\infty}(\mathbb{R}^d)$.

\item Given $1\leq p,q\leq \infty$, the mixed space-time norms are defined as follows
\begin{equation*}
\|f\|_{L^p_x L^q_T}=\left(\int_{\mathbb{R}} \left(\int_0^T |f(x,t)|^q dt \right)^{p/q}\, dx \right)^{1/p},      
\end{equation*}
and
\begin{equation*}
\|f\|_{L^p_T L^q_x}=\left(\int_0^T\left(\int_{\mathbb{R}} |f(x,t)|^q dx \right)^{p/q}\, dt \right)^{1/p}     
\end{equation*}
with the standard modification in the cases $p=\infty$ or $q=\infty$. We also work with mixed spaces such as $L^{q}_T H^s=L^{q}_T H^s_x$, $1\leq q \leq \infty$. 
\end{itemize}


\section{Preliminaries}\label{Secprelimin}

In this section, we introduce some preliminary results necessary for the deduction of our main results. Firstly, we mention some key results on classical pseudo-differential operators and useful consequences. Finally, we present some additional commutator estimates.

\subsection{Pseudo-differential calculus} The deduction of our main theorems depends on the theory of classical pseudo-differential operators. Although we work on spatial dimension one, we present the results in this part for an arbitrary spatial dimension. 

Let $m\in\mathbb{R}$, and  let $a(x,\xi)\in \mathcal{C}^{\infty}(\mathbb{R}^d \times \mathbb{R}^d)$. We say that \emph{$a$ is a symbol of order $m$} if for all multi-indices $\alpha,\beta\in\mathbb{N}^d$ there exists a positive constant $C_{\alpha,\beta}$ such that 
    \begin{equation*}
    \label{symbol}
        |\partial_x^{\alpha} \partial_{\xi}^{\beta} a(x,\xi)| \leq C_{\alpha,\beta} \langle \xi \rangle^{m-|\beta|} ~~ \text{ for all } x,\xi\in\mathbb{R}^d,
    \end{equation*}
where $\langle \xi \rangle = (1+|\xi|^2)^{1/2}$. The set $\mathcal{S}^m(\mathbb{R}^d \times \mathbb{R}^d)=\mathcal{S}^m$ consists of all the symbols of order $m$. A pseudo-differential operator is a mapping $f\mapsto \Psi_a f$ given by $$\Psi_a f(x) = (2\pi)^{-d} \int_{\mathbb{R}^d} e^{i\langle x,\xi \rangle} a(x,\xi) \widehat{f}(\xi) d\xi,$$ where $a(x,\xi)\in\mathcal{S}^m$ is the symbol of $\Psi_a$ for some $m\in\mathbb{R}$. In this case, we say that $\Psi_{a}\in \mathrm{OP}\mathcal{S}^m$.

\begin{remark}
In this work, we consider estimates with weighted functions $x\mapsto \psi(x)$, where $\psi\in C^{\infty}(\mathbb{R}^d)$ is bounded with bounded derivatives of any order. This implies that $\psi\in \mathcal{S}^0$. Since $\psi$ does not depend on the frequency variable $\xi$, for simplicity in the notation, we denote its symbol and corresponding operator by $\psi$.    
\end{remark}

\begin{lemma}\label{LemmaAllpseudo}
Let $m,m_1,m_2\in \mathbb{R}$. 
    \begin{itemize}
        \item[(i)]  If $m_1\leq m_2$, then $\mathcal{S}^{m_1} \subseteq \mathcal{S}^{m_2}$. 
        \item[(ii)] Let $\Psi_a\in \mathrm{OP}\mathcal{S}^0$ (i.e., $a\in \mathcal{S}^0)$. Then the operator $\Psi_a$ initially defined on $S(\mathbb{R}^d)$, extends to a bounded operator from $L^2(\mathbb{R}^d)$ to itself.
        \item[(iii)] Let $1<p<\infty$, $\Psi_a\in \mathrm{OP}\mathcal{S}^m$ (i.e., $a\in \mathcal{S}^m)$, then
\begin{equation*}
    \|\Psi_a(f)\|_{L^p}\lesssim \|J^m f\|_{L^p}.
\end{equation*}
\item[(iv)]     Let $a\in\mathcal{S}^{m_1}$ and $b\in\mathcal{S}^{m_2}$. Then, there exists a symbol $c\in\mathcal{S}^{m_1+m_2}$ such that $\Psi_c = \Psi_a \circ \Psi_b  \in \mathrm{OP}\mathcal{S}^{m_1+m_2}$. Moreover, 
\begin{equation*}
  c\approx \sum_{\mu} \frac{(-i)^{|\mu|}}{\mu!} (\partial_\xi^{\mu} a)(\partial_x^{\mu} b)  
\end{equation*}
in the sense that $$c - \sum_{|\mu| < N} \frac{(-i)^{|\mu|}}{\mu!} (\partial_\xi^{\mu} a)(\partial_x^{\mu} b) \in \mathcal{S}^{m_1+m_2-N}$$ for all integers $N>0$.

\item[(v)]     Let $a\in\mathcal{S}^{m_1}$ and $b\in\mathcal{S}^{m_2}$. Then the symbol of the commutator $[\Psi_a,\Psi_b] = \Psi_a \Psi_b - \Psi_b \Psi_a\in \mathrm{OP}\mathcal{S}^{m_1+m_2-1}$ is given by 
\begin{equation*}
  c = \frac{1}{i} \sum_{j=1}^{d} \frac{\partial a}{\partial \xi_j} \frac{\partial b}{\partial x_j} - \frac{\partial a}{\partial x_j} \frac{\partial b}{\partial \xi_j} \mod{\mathcal{S}^{m_1+m_2-2}}.  
\end{equation*}
\end{itemize}
\end{lemma}
\begin{proof}
The proofs of these results can be consulted in \cite{stein1993harmonic,Hormander2007}.  
\end{proof}

Next, we recall some useful properties of pseudo-differential operators and fractional operators $|D|^sJ^s$, when they are applied to functions with separated supports.

\begin{lemma}\label{IneqSupports}    
Let $\alpha$ be a multi-index, $\Psi_a \in \mathrm{OP}\mathcal{S}^m$, and let $p\in [2,\infty]$. If $f\in L^2(\mathbb{R}^d)$ and $g\in L^p(\mathbb{R}^d)$ satisfy $$\dist(\supp(f),\supp(g)) \geq \delta > 0,$$ then $$\|g \partial_x^{\alpha} \Psi_a f\|_{L^2(\mathbb{R}^d)} \lesssim \|g\|_{L^p(\mathbb{R}^d)} \|f\|_{L^2(\mathbb{R}^d)}.$$
\end{lemma}
\begin{proof}
We refer to \cite{KenigLinaresPonceVega2018} and \cite{Mendez2024,FreireMendezRiano2022}.
\end{proof}

By using the above lemma, it is obtained a fractional derivative version of it with the operators $|D|^s$ and $J^s$. 

\begin{lemma}
\label{Ds-Js}
    Let $s\in\mathbb{R}$ and $s_1 \in (0,1)$. If $f \in L^{\infty}(\mathbb{R}^d)$ and $g\in L^2(\mathbb{R}^d)$ satisfy $$\dist (\supp(f) , \supp(g) ) \geq \delta > 0,$$ then $$\|f \, |D|^{s_1} J^s g\|_{L^2(\mathbb{R}^d)} \lesssim \|f\|_{L^\infty(\mathbb{R}^d)} \|g\|_{L^2(\mathbb{R}^d)}.$$

\end{lemma}
\begin{proof} 
We refer to \cite[Lemma 2.6]{FreireMendezRiano2022}. 
\end{proof}

In the following lemma, we give formulas connecting the propagation of regularity effect in different domains. This type of estimates can be found in \cite{KenigLinaresPonceVega2018} and \cite{Mendez2024}, and the version presented here in taken from \cite{FreireMendezRiano2022}.

\begin{lemma} \label{LocReg}
Let $f \in H^{-m}(\mathbb{R}^d)$ for some integer $m\geq 0$, and let $\theta_1,\theta_2 \in \mathcal{C}^{\infty}(\mathbb{R}^d) \setminus \{0\}$ be such that $0\leq \theta_1 , \theta_2 \leq 1$, $$\dist(\supp(1 - \theta_1), \supp(\theta_2)) \geq \delta > 0,$$ and satisfying $\partial^{\gamma} \theta_1 , \partial^{\gamma} \theta_2 \in L^{\infty}(\mathbb{R}^d)$ for each multi-index $\gamma$.
\begin{itemize}
\item[(i)] If $\beta \in [0,2)$ and $\theta_1 f, \theta_1 |D|^{\beta} f \in L^2(\mathbb{R}^d)$, then $$\|\theta_2 J^{\beta} f\|_{L^2} \lesssim \|\theta_1 f\|_{L^2} + \|\theta_1 |D|^{\beta} f\|_{L^2} + \|J^{-m}f\|_{L^2},$$ and so $\theta_2 J^{\beta} f \in L^2(\mathbb{R}^d)$.

\item[(ii)] If $\beta \in [0,2)$ and $\theta_1 J^{\beta} f \in L^2(\mathbb{R}^d)$, then $$\|\theta_2 f\|_{L^2} + \|\theta_2 |D|^{\beta} f\|_{L^2} \lesssim \|\theta_1 J^{\beta} f\|_{L^2} + \|J^{-m}f\|_{L^2},$$ and so $\theta_2 f, \theta_2 |D|^{\beta} f \in L^2(\mathbb{R}^d)$.

\item[(iii)] If $s>0$, $r\in [0,s]$, and $\theta_1 J^{s} f \in L^2(\mathbb{R}^d)$, then $$\|\theta_2 J^{r} f\|_{L^2} \lesssim \|\theta_1 J^{s} f\|_{L^2} + \|J^{-m}f\|_{L^2},$$ and so $\theta_2 J^{r} f \in L^2(\mathbb{R}^d)$.

\item[(iv)] If $s>0$ and $\theta_1 J^{s} f \in L^2(\mathbb{R}^d)$, then $$\|J^{s} (\theta_2 f)\|_{L^2} \lesssim \|\theta_1 J^{s} f\|_{L^2} + \|J^{-m}f\|_{L^2},$$ and so $J^{s} (\theta_2 f) \in L^2(\mathbb{R}^d)$.
  \end{itemize}
\end{lemma}

\begin{proposition}\label{ExpDs-Js}
Let $s>0$. Then 
 \begin{equation}\label{conmexpansionJD}
    |D|^s=J^s - \sum_{j=1}^{\infty} \binom{s/2}{j} (-1)^{j+1} J^{-(2j-s)}, 
    \end{equation}
where the above identity is justified in $H^s(\mathbb{R}^d)$. Moreover, 
    \begin{equation}\label{estimateJD}
\|(J^s-|D|^s)f\|_{L^2(\mathbb{R}^d)}\lesssim \|f\|_{H^{s-2}(\mathbb{R}^d)}, 
\end{equation}
for all $f\in H^{s-2}(\mathbb{R}^{d})$, i.e., the operator $(J^s-|D|^s)$ admits a continuous extension to $H^{s-2}(\mathbb{R}^d)$. More generally, given $M\geq 0$ integer, it follows
\begin{equation}\label{estimateJDGen}
 \begin{aligned}
\Big\|\Big(J^s-|D|^s-\sum_{j=1}^M \binom{s/2}{j} (-1)^{j+1} J^{-(2j-s)} \Big)f\Big\|_{L^2(\mathbb{R}^d)}\lesssim \|f\|_{H^{s-2(M+1)}(\mathbb{R}^d)},
 \end{aligned}   
\end{equation}
for all $f\in H^{s-2(M+1)}(\mathbb{R}^d)$,  where we adopt the convention that an empty sum equals zero, e.g., $\sum_{j=1}^0(\dots)=0$.
\end{proposition}

\begin{proof}
We use the binomial expansion (e.g., \cite{BourgainLi2014}) to get 
    \begin{equation}\label{Bexpansion1}
    \begin{aligned}
        (\langle \xi \rangle^{s} - |\xi|^{s}) = \langle \xi \rangle^{s} (1 - \langle \xi \rangle^{-s} |\xi|^{s})&= \langle \xi \rangle^{s} (1 - (1-\langle \xi \rangle^{-2})^{s/2})\\
        &= \langle \xi \rangle^{s} \left(1 - \sum_{j=0}^{\infty} \binom{s/2}{j} \frac{(-1)^j}{\langle \xi \rangle^{2j}} \right)\\
        &= \sum_{j=1}^{\infty} \binom{s/2}{j} \frac{(-1)^{j+1} \langle \xi \rangle^{s}}{\langle \xi \rangle^{2j}}.
    \end{aligned}
    \end{equation}
Notice that the previous expression makes sense for each $\xi\in \mathbb{R}^d$, because the series $\sum_{j=1}^{\infty} \binom{\alpha}{j}$ converges absolutely for any real number $\alpha>0$.

On the other hand, since $\sup_{\xi\in \mathbb{R}^d}\frac{1}{\langle \xi \rangle^{\beta}}\leq 1$, we have $\|J^{-\beta}f\|_{L^2}\leq \|f\|_{L^2}$, for all $\beta \geq 0$. It follows
    \begin{align*}
        \|(J^s - |D|^s) f\|_{L^2} \leq& \sum_{j=1}^{\infty} \left|\binom{s/2}{j}\right| \|J^{-(2j-s)}f\|_{L^2}\\ =& \sum_{j=1}^{\infty} \left|\binom{s/2}{j}\right| \|J^{2-2j}J^{s-2}f\|_{L^2}\\
        \leq&\left( \sum_{j=1}^{\infty} \left|\binom{s/2}{j}\right| \right) \|J^{s-2}f\|_{L^2} \lesssim \|f\|_{H^{s-2}}.
    \end{align*}
The above inequality shows that \eqref{conmexpansionJD} is valid in $H^s(\mathbb{R}^{d})$, and by density, the operator $J^s-|D|^s$ can be extended to $H^{s-2}(\mathbb{R}^d)$. A similar argument yields \eqref{estimateJDGen}. 
\end{proof}


\subsection{Commutator estimates} 

To establish bounds for the nonlinearities in \eqref{Benjamineq}, we consider some well-known estimates for commutators.

\begin{lemma} \label{Kato-Ponce}
    If $s>0$ and $p\in (1,\infty)$, then $$\|[J^s,f]g\|_{L^p(\mathbb{R}^d)} \lesssim \|\nabla f\|_{L^{\infty}(\mathbb{R}^d)} \|J^{s-1} g\|_{L^{p}(\mathbb{R}^d)} + \|J^{s} f\|_{L^{p}(\mathbb{R}^d)}\|g\|_{L^{\infty}(\mathbb{R}^d)}.$$
\end{lemma}
\begin{proof}
This inequality is a particular case of the one deduced by Kato and Ponce in \cite{KatoPonce1988}.
\end{proof}

\begin{lemma} \label{Grafakos-Oh}
Let $s>0$, $1<p<\infty$ and $p_1,p_2,q_1,q_2\in (1,\infty]$ such that $\frac{1}{p}=\frac{1}{p_1} + \frac{1}{q_1} =\frac{1}{p_2} + \frac{1}{q_2} $, then 
\begin{equation*}
   \||D|^s(fg)\|_{L^{p}(\mathbb{R}^d)} \lesssim \||D|^s f\|_{L^{p_1}(\mathbb{R}^d)} \|g\|_{L^{q_1}(\mathbb{R}^d)} + \|f\|_{L^{p_2}(\mathbb{R}^d)}\||D|^s g\|_{L^{q_2}(\mathbb{R}^d)} 
\end{equation*}
and
\begin{equation*}
\|J^s(fg)\|_{L^{p}(\mathbb{R}^d)} \lesssim \|J^s f\|_{L^{p_1}(\mathbb{R}^d)} \|g\|_{L^{q_1}(\mathbb{R}^d)} + \|f\|_{L^{p_2}(\mathbb{R}^d)}\|J^s g\|_{L^{q_2}(\mathbb{R}^d)}.    
\end{equation*}
\end{lemma}
\begin{proof}
    We refer to \cite{GrafakosOh2014}.
\end{proof}

\begin{lemma} \label{CommutatorKLPV}
    Let $\phi\in\mathcal{C}^{\infty}(\mathbb{R})$ with $\phi'\in\mathcal{C}_0^{\infty}(\mathbb{R})$. If $s\in\mathbb{R}$, then for any $l>|s-1|+1/2$ $$\|[J^s,\phi]f\|_{L^2} + \|[J^{s-1},\phi]\partial_x f\|_{L^2} \leq c\|J^l \phi'\|_{L^2} \|J^{s-1} f\|_{L^2}.$$
\end{lemma}
\begin{proof}
    This inequality can be found in \cite{KenigLinaresPonceVega2018}. 
\end{proof}


\section{Kato's smoothing effect: proof of Theorem \ref{SmoothingEffectThrm}}\label{Katosection}

In this section, we deduce a smoothing effect for solutions of \eqref{Benjamineq} similar to that established by Kato \cite{Kato1983} for the KdV equation. Our main tools come from the theory of pseudo-differential operators.

{\bf Equivalent model}. For simplicity, let us rewrite the operators in the equation in \eqref{Benjamineq} in terms of the $J^{l}$ operators. Notice that $\mathcal{H} \partial_x^2 = \partial_x |D|$. Furthermore, using  the symbols of the operators $\partial_x^{2N}$ and $J^{2N}$ in frequency, we have that the following identity holds
\begin{equation*}
(-1)^{N+1}\partial_x^{2N} = - J^{2N} + \sum_{j=0}^{N-1}\binom{N}{j}(-1)^j \partial_x^{2j}.    
\end{equation*}
It follows that the equation in \eqref{Benjamineq} can be rewritten as 
\begin{equation*} 
    \partial_t u + \gamma \partial_x |D| u - \partial_x J^{2N} u + P_1(D) u + \partial_x u +\sum_{k=1}^{M}b_k u^k\partial_x u = 0, \quad x,t \in \mathbb{R},
\end{equation*}
where $P_1(D)=\sum_{k=1}^{N-2} \widetilde{a_k} \partial_x^{2k+1}=i\sum_{k=1}^{N-2}(-1)^k\widetilde{a_k} D^{2k+1}$. We emphasize that if $N\in\{1,2\}$ or $a_k=0$ for every $k\in \{1,\ldots,N-2\}$, then $P_1(D)=0$.

One may repeat this procedure for each one of the derivatives in $P_1(D)$, obtaining the following form of \eqref{Benjamineq}
\begin{equation} \label{Benj2}
    \begin{cases}
    \partial_t u + \gamma \partial_x |D| u - \partial_x J^{2N} u + Q(J) u + \alpha \partial_x u + \sum_{k=1}^{M}b_k u^k\partial_x u = 0, \quad x,t \in \mathbb{R},\\
    u(x,0) = u_0(x),
\end{cases} 
\end{equation}
where we assume:
\begin{itemize}
    \item If $N=1$, then $Q(J) = 0$.
    \item If $N\geq 2$, then $Q(J) = \sum_{j=1}^{N-1} c_j \partial_x J^{2j} = \sum_{j=0}^{N-2} c_{j}' \partial_x J^{2N - 2 - 2j}$, where the $c_j$'s are real numbers and $c_j' = c_{N-1-j}$ for any $j$.
    \item $\gamma,\alpha \in \mathbb{R}$. 
\end{itemize}
The above equation is convenient to exploit the fact that $J^s$ is a pseudo-differential operator for any $s\in \mathbb{R}$.

{\bf Weighted functions}. Let $A$ be a positive real number. To prove the desired smoothing effect, we consider a sequence of smooth functions $(\psi_\ell)_{\ell\geq 1}$ such that $0\leq \psi_{\ell}'\leq 1$,  $\psi_{\ell}' \equiv 1$ in the interval $[-2^{-\ell}A , 2^{-\ell} A]$ and $\supp(\psi_{\ell}') \subseteq [-2^{-(\ell-1/2)}A , 2^{-(\ell-1/2)} A]$ for all $\ell\geq 1$.  Hence, the sequence of compact sets $\{\supp(\psi_{\ell}')\}_{\ell\geq 1}$ is monotone decreasing according to the inclusion relation, and the distance between the supports of $1-\psi_{\ell}'$ and $\psi_{\ell+1}'$ is a positive constant for every $\ell\geq 1$.

{\bf Strategy of the proof}. Using energy estimates, we aim to show that the solutions of \eqref{Benj2} gain locally $N$ derivatives, which will be done in $N$ steps. More precisely, we prove that the solutions of \eqref{Benj2} gain one local derivative in the first step.  In the second step, we use this first step to show that the solution also gains a second local derivative, and so on up to $N$ derivatives. We remark that to take advantage of the previous step, in each phase we work with the support of the weighted function $\psi_{\ell}'$ that is contained in that of $\psi_{\ell-1}'$. This is the reason behind the construction of the particular sequence $(\psi_{\ell})_{\ell\geq 1}$.

\begin{proof}[Proof of Theorem \ref{SmoothingEffectThrm}]

Using continuous dependence and the fact that the lifetime of solutions in Lemmas \ref{LWPN12} and \ref{StandardLWP} decreases with the norm, we will assume that the solution $u$ of \eqref{Benjamineq} is regular enough to justify all the arguments below.  In particular, we work with the equivalent equation in \eqref{Benj2}. To observe this, we consider  $\varphi$ being a nonnegative smooth function with compact support such that $\supp(\varphi)\subseteq (-1,1)$ and $\int_{\mathbb{R}} \varphi(x)dx = 1$. For each $\lambda>0$, we define the functions $\varphi_{\lambda}(x) = \lambda^{-1} \varphi(\lambda^{-1} x)$ and $u_0^{\lambda} = \varphi_\lambda \ast u_0$. It follows that each $u_0^\lambda$ lies in $H^{\infty}(\mathbb{R})$, $\|u^{\lambda}_0 - u_0\|_{H^s} \to 0$ as $\lambda\to 0$, and $\|u_0^{\lambda}\|_{H^s} \leq  \|u_0\|_{H^s}$. Consequently, if $u\in C([0,T];H^s(\mathbb{R}))$,  we have that for all $\lambda>0$,  $u^{\lambda}\in C([0,T];H^{\infty}(\mathbb{R}))$ and by continuous dependence 
\begin{equation*}
\sup_{t\in [0,T]}\|u^{\lambda}(t) - u(t)\|_{H^s}\to 0, \quad \|\partial_x (u^{\lambda})-\partial_x u\|_{L^4_T L^{\infty}_x} \to 0,    
\end{equation*}
as $\lambda\to 0$. As a consequence, we can carry out the estimates below for the functions $u^{\lambda}$ and then take $\lambda \to 0$ to obtain the desired result for $u$. For a more detailed idea, see \cite{IsazaLinaresPonce2015}.

Let $R>0$ be such that
\begin{equation*}
\sup_{t\in [0,T]} \|u(\cdot,t)\|_{H^s(\mathbb{R})} \leq R.    
\end{equation*}
To begin with, we shall prove that 
\begin{equation}\label{Inducteq}
  \int_0^T \int_{\mathbb{R}} (J^{s+\ell} u)^2 \, \psi_{\ell}' \, dx \, dt  \leq c(\|u_0\|_{H^s};T)<\infty  
\end{equation}
for every positive integer $\ell\leq N$, which will be done by induction on $\ell\geq 1$. Applying the operator $J^{s-N+\ell}$ to both sides of equation \eqref{Benj2}, multiplying by $\psi_{\ell} J^{s-N+\ell} u$, and integrating with respect to the $x$-variable, we get
    \begin{equation}\label{energuestimateSm} 
    \begin{aligned}
    &\underbrace{\int_{\mathbb{R}} \psi_{\ell} \, J^{s-N+\ell} u \, \partial_t J^{s-N+\ell} u \, dx}_{A_1} + \underbrace{\gamma \int_{\mathbb{R}} \psi_{\ell} \, J^{s-N+\ell} u \, \partial_x |D| J^{s-N+\ell} u \, dx}_{A_2}\\ &\underbrace{- \int_{\mathbb{R}} \psi_{\ell} \, J^{s-N+\ell} u \, \partial_x J^{s+N+\ell} u\, dx}_{A_3} + \underbrace{\int_{\mathbb{R}} \psi_{\ell} \, J^{s-N+\ell} u \, Q(J) J^{s-N+\ell} u\, dx}_{A_4} \\& + \underbrace{\alpha \int_{\mathbb{R}}\psi_{\ell} \, J^{s-N+\ell} u \, (\partial_x J^{s-N+\ell} u)\, dx}_{A_5}  + \sum_{k=1}^M\underbrace{b_{k}\int_{\mathbb{R}} \psi_{\ell} \, J^{s-N+\ell} u \, J^{s-N+\ell} (u^{k}\partial_x u)\, dx}_{A_{6k}} = 0.
    \end{aligned}
    \end{equation}
We start with the base step, i.e., we estimate each of the terms in \eqref{energuestimateSm} for $\ell=1$.
\\ \\
\underline{\bf Case $\ell=1$ in \eqref{energuestimateSm}}. Notice that 
    \begin{equation}\label{A1}
        A_1 = \frac{1}{2} \int_{\mathbb{R}} \psi_1 \, \partial_t (J^{s-N+1} u)^2 \, dx= \frac{1}{2} \frac{d}{dt} \int_{\mathbb{R}} (J^{s-N+1} u)^2 \, \psi_1 \, dx.
    \end{equation}
Setting $K:=|D| - J$, we write
    \begin{equation*}
        \begin{aligned}
            A_2 &= \gamma \int_{\mathbb{R}} \psi_1 \, J^{s-N+1} u \, \partial_x K J^{s-N+1} u \, dx + \gamma \int_{\mathbb{R}} \psi_1 \, J^{s-N+1} u \, \partial_x J^{s-N+2} u \, dx\\
            &=: A_{21} + A_{22}.
        \end{aligned}
    \end{equation*}
 Applying the Cauchy-Schwarz inequality and \eqref{estimateJD} in Proposition \ref{ExpDs-Js}, we obtain
    \begin{equation} \label{A21}
        |A_{21}| \lesssim \|\psi_1 J^{s-N+1} u\|_{L^2} \|\partial_x K J^{s-N+1} u\|_{L^2} \lesssim \|J^{s-N+1} u\|_{L^2}^2 \lesssim R^2.
    \end{equation}
To estimate $A_{2,2}$, let us consider two cases $N=1$ and $N\geq 2$.

\underline{Assume that $N=1$}. Using that $J$ is a self-adjoint operator, we integrate by parts to get 
\begin{equation*}
\begin{aligned}
A_{22}=&\gamma\int_{\mathbb{R}} \psi_1 J^{s} u \partial_x J^{s+1} u \, dx\\
=& \gamma \int_{\mathbb{R}} [J,\psi_1] J^s  u  \partial_x J^{s}u\, dx+\gamma \int_{\mathbb{R}} \psi_1 J^{s+1}  u \partial_x J^{s}u\, dx \\
=& \gamma \int_{\mathbb{R}} [J,\psi_1] J^s  u  \partial_x J^{s}u\, dx-\gamma \int_{\mathbb{R}} \psi_1' J^{s+1}  u J^{s}u\, dx-A_{22},
\end{aligned}    
\end{equation*}
which yields
\begin{equation*}
\begin{aligned}
A_{22}=& \frac{\gamma}{2} \int_{\mathbb{R}} [J,\psi_1] J^s  u  \partial_x J^{s}u\, dx-\frac{\gamma}{2} \int_{\mathbb{R}} \psi_1' J^{s+1}  u J^{s}u\, dx\\
=:&A_{221}+A_{222}.
\end{aligned}    
\end{equation*}
Using Lemma \ref{LemmaAllpseudo} (v) we have
\begin{equation*}
\begin{aligned}
\left[J,\psi_1\right]=-\psi_1'\partial_x J^{-1} + \Phi_{-1},  
\end{aligned}    
\end{equation*}
where $\Phi_{-1}\in \mathrm{OP}\mathcal{S}^{-1}$. Then, integration by parts and the fact that $\partial_x^2=1-J^2$ yield
\begin{equation*}
 \begin{aligned}
A_{221}=&-\frac{\gamma}{2}\int_{\mathbb{R}} \psi'_1 \partial_x J^{s-1} u \partial_x J^s u\, dx +\frac{\gamma}{2}\int_{\mathbb{R}}  \Phi_{-1} J^s  u  \partial_x J^{s}u\, dx \\
=&\frac{\gamma}{2}\int_{\mathbb{R}} \psi''_1 \partial_x J^{s-1} u J^s u\, dx+\frac{\gamma}{2}\int_{\mathbb{R}} \psi'_1 \partial_x^2 J^{s-1} u J^s u\, dx -\frac{\gamma}{2}\int_{\mathbb{R}}  \partial_x \Phi_{-1} J^s  u  J^{s}u\, dx\\
=&\frac{\gamma}{2}\int_{\mathbb{R}} \psi''_1 \partial_x J^{s-1} u J^s u\, dx+\frac{\gamma}{2}\int_{\mathbb{R}} \psi'_1  J^{s-1} u J^s u\, dx-\frac{\gamma}{2}\int_{\mathbb{R}} \psi'_1 J^{s+1} u J^s u\, dx \\
&-\frac{\gamma}{2}\int_{\mathbb{R}}  \partial_x \Phi_{-1} J^s  u  J^{s}u\, dx\\
=:& \sum_{j=1}^4 A_{221j}.
 \end{aligned}   
\end{equation*}
Since $\partial_x \Phi_{-1}\in \mathrm{OP}\mathcal{S}^{0}$, Lemma \ref{LemmaAllpseudo} (iii) shows
\begin{equation*}
\begin{aligned}
\Big|\sum_{\substack{1\leq j \leq 4\\ j\neq 3}} A_{221j}\Big|\lesssim \|J^s u\|_{L^2}^2\lesssim R^2.    
\end{aligned}    
\end{equation*}
Using Cauchy-Schwarz inequality and Young's inequality with $\epsilon>0$, we get
\begin{equation*}
 \begin{aligned}
|A_{2213}|\lesssim \|J^s u\|_{L^2} \|(\psi_1')^{\frac{1}{2}}J^{s+1} u\|_{L^2}  \leq c_{\epsilon}R^2+\epsilon \|(\psi_1')^{\frac{1}{2}}J^{s+1} u\|_{L^2}^2,  
 \end{aligned}   
\end{equation*}
for some constant $c_{\epsilon}>0$. The idea is to choose $\epsilon>0$ small enough to absorb the last term on the right side of the previous inequality by the smoothing effect provided by term $A_3$, which corresponds to the higher-order dispersion. The estimate for $A_{222}$ is similar to that of $A_{2213}$. Summarizing, we deduce that for any $\epsilon>0$,
\begin{equation}\label{N1caseeq}
    |A_{22}|\leq c_{\epsilon} R^2+\epsilon \|(\psi_1')^{\frac{1}{2}}J^{s+1} u\|_{L^2}^2,
\end{equation}
for some $c_{\epsilon}>0$. 

\underline{Assume that $N\geq 2$}. In this case, the estimate of $A_{22}$ is simpler than the one made in the case $N=1$. By the integration by parts formula,
    \begin{equation} \label{A22}
        \begin{aligned}
            |A_{22}| &= \left|- \gamma \int_{\mathbb{R}} \psi_1' \, J^{s-N+1} u \, J^{s-N+2} u \, dx - \gamma \int_{\mathbb{R}} \psi_1 \, (\partial_x J^{s-N+1} u) \, J^{s-N+2} u \, dx \right| \\
            &\lesssim \|\psi_1' \, J^{s-N+1} u\|_{L^2} \|J^{s-N+2} u\|_{L^2} + \|\psi_1 \, \partial_x J^{s-N+1} u\|_{L^2} \|J^{s-N+2} u\|_{L^2}\\
            &\lesssim R^2.
        \end{aligned}
    \end{equation}
From \eqref{A21}, \eqref{N1caseeq} and \eqref{A22}, it follows that for each $\epsilon>0$, there exists $c_{\epsilon}>0$ such that 
\begin{equation} \label{A2}
|A_2| \leq c_{\epsilon}R^2+\epsilon \|(\psi_1')^{\frac{1}{2}}J^{s+1} u\|_{L^2}^2.
    \end{equation}
Now, we deal with $A_3$. Since $J^N$ is a self-adjoint operator, after applying integration by parts, we have 
    \begin{equation} \label{A3in}
        \begin{aligned}
            A_3 &= - \int_{\mathbb{R}} J^N(\psi_1 J^{s-N+1} u) \, \partial_x J^{s+1}u \, dx\\
            &= - \int_{\mathbb{R}} ([J^N,\psi_1] J^{s-N+1} u)\, \partial_x J^{s+1}u \, dx - \int_{\mathbb{R}} \psi_1\, J^{s+1} u \, \partial_x J^{s+1}u \, dx \\
            &= - \int_{\mathbb{R}} ([J^N,\psi_1] J^{s-N+1} u)\, \partial_x J^{s+1}u \, dx + \frac{1}{2} \int_{\mathbb{R}} (J^{s+1} u)^2 \psi_1' dx\\
            &=: A_{31} + \frac{1}{2} \int_{\mathbb{R}} (J^{s+1} u)^2 \psi_1' dx.
        \end{aligned}
    \end{equation}
Note that the last term on the right-hand side of the above identity is a smoothing effect provided by the dispersion of higher-order. By Lemma \ref{LemmaAllpseudo} (iv), there exists a pseudo-differential operator $\Psi_{N-3}$ of order $N-3$ such that 
    \begin{equation} \label{JN exp}
        [J^N,\psi_1] = - N \psi_1' \partial_x J^{N-2} - \frac{N(N-1)}{2} \psi_1'' J^{N-2} + \frac{N(N-2)}{2}\psi_1'' J^{N-4} + \Psi_{N-3}.
    \end{equation}
    Thus, 
    \begin{equation} \label{A31}
        \begin{aligned}
            A_{31} =& N \int_{\mathbb{R}} \psi_1' \, (\partial_x J^{s-1}u) \, \partial_x J^{s+1}u \, dx + \frac{N(N-1)}{2} \int_{\mathbb{R}} \psi_1'' \, J^{s-1}u \, \partial_x J^{s+1}u \, dx\\
            &- \frac{N(N-2)}{2}\int_{\mathbb{R}} \psi_1'' \, J^{s-3}u \, \partial_x J^{s+1}u \, dx - \int_{\mathbb{R}} (\Psi_{N-3} J^{s-N+1}u)\, \partial_x J^{s+1}u \, dx\\
            =:& A_{311} + A_{312} + A_{313} + A_{314}.
        \end{aligned}
    \end{equation}
    Applying integration by parts and the identity $J^2 = 1 - \partial_x^2$, we get
    \begin{equation} \label{A311}
        \begin{aligned}
            A_{311} =& -N \int_{\mathbb{R}} \psi_1'' \, (\partial_x J^{s-1}u) \, J^{s+1}u \, dx - N \int_{\mathbb{R}} \psi_1' \, (\partial_x^2 J^{s-1}u) \, J^{s+1}u \, dx\\
            =& - N \int_{\mathbb{R}} ([J,\psi_1''] \, \partial_x J^{s-1}u) \, J^{s}u \, dx - N \int_{\mathbb{R}} \psi_1'' \, (\partial_x J^{s}u) \, J^{s}u \, dx\\ &- N \int_{\mathbb{R}} \psi_1' \, (J^{s-1}u) \, J^{s+1}u \, dx + N \int_{\mathbb{R}} \psi_1' \, (J^{s+1}u)^2 \, dx\\
            =:& A_{3111} + A_{3112} + A_{3113} + N \int_{\mathbb{R}} (J^{s+1}u)^2 \, \psi_1' \, dx.
        \end{aligned}
    \end{equation}
    Once again, we observe a smoothing effect from the above identity for $A_{311}$. By Lemma \ref{LemmaAllpseudo}, we have that $[J,\psi'']\partial_x J^{-1} \in \mathrm{OP}\mathcal{S}^{0}$, and it determines a bounded operator on $L^2(\mathbb{R})$. Therefore,
    \begin{equation} \label{A3111}
        |A_{3111}| \lesssim \|[J,\psi_1''] \, \partial_x J^{-1} J^{s}u\|_{L^2} \|J^s u\|_{L^2} \lesssim  \|J^s u\|_{L^2}^2 \leq R^2.
    \end{equation}
    By using the commutator $[J,\psi_1']$ and arguing as before, we obtain 
    \begin{equation} \label{A3113}
        |A_{3113}| \lesssim R^2.
    \end{equation}
    Observe that $A_{3112} = \frac{N}{2} \int_{\mathbb{R}} \psi_1'''\,(J^s u)^2 \, dx,$ thus
    \begin{equation} \label{A3112}
        |A_{3112}| \lesssim R^2.
    \end{equation}
    This completes the estimate of $A_{311}$.
    
    Given that $\partial_x J$ is a skew-adjoint operator, we have
    \begin{equation*} 
        \begin{aligned}
            A_{312} =& - \frac{N(N-1)}{2} \int_{\mathbb{R}} \partial_x J (\psi_1'' J^{s-1}u) \, J^s u \, dx\\
            =& - \frac{N(N-1)}{2} \int_{\mathbb{R}} ( [\partial_x J,\psi_1''] J^{s-1}u) \, J^s u \, dx - \frac{N(N-1)}{2} \int_{\mathbb{R}} \psi_1'' \, (\partial_x J^{s}u) \, J^s u \, dx\\
            =& - \frac{N(N-1)}{2} \int_{\mathbb{R}} ( [\partial_x J,\psi_1''] J^{s-1}u) \, J^s u \, dx + \frac{N(N-1)}{4} \int_{\mathbb{R}} \psi_1''' \, (J^{s}u)^2\, dx.
        \end{aligned}
    \end{equation*}
    Since $[\partial_x J,\psi_1'']J^{-1}$ has order $0$, we can use the Cauchy-Schwarz inequality to estimate the first term on the right-hand side of the above equation by a constant times $R^2$. While the second term clearly admits the same kind of bound. We obtain
    \begin{equation} \label{A312}
        |A_{312}| \lesssim R^2.
    \end{equation}
    Decomposing the term $A_{313}$ by means of the commutator $[\partial_x J,\psi_1'']$, and using the same previous ideas, we get 
    \begin{equation} \label{A313}
        |A_{313}| \lesssim R^2.
    \end{equation}
  Given that  $\partial_x J$ is skew-adjoint, and that $\partial_x J \Psi_{N-3} J^{-N+1} \in \mathrm{OP}\mathcal{S}^{0}$, we deduce
    \begin{equation}\label{A314} 
        |A_{314}| =\Big| \int_{\mathbb{R}} (\partial_x J \Psi_{N-3} J^{-N+1} J^s u) \, J^s u \, dx\Big|\lesssim R^2.
    \end{equation}
    Collecting \eqref{A3in}-\eqref{A314}, we conclude that 
    \begin{equation}\label{A3}
        A_3 = \left( N + \frac{1}{2} \right)\int_{\mathbb{R}} (J^{s+1} u)^2 \, \psi_1' \, dx + \widetilde{A_3},
    \end{equation}
    where $|\widetilde{A_3}| \lesssim R^2$.

    For $A_4$, notice that 
    \begin{equation} \label{A4in}
        \begin{aligned}
            A_4 = \sum_{j=0}^{N-2} c_j' \int_{\mathbb{R}} \psi_1 \, J^{s-N+1} u \, \partial_x J^{s+N-1-2j} u\, dx=: \sum_{j=0}^{N-2} A_{4j}.
        \end{aligned}
    \end{equation}
    For any $j\in \{0,\ldots, N-2\}$, we have 
    \begin{equation} \label{A4dec}
        \begin{aligned}
            A_{4j} &= -c_j' \int_{\mathbb{R}} ([\partial_x J^{N-1}, \psi_1] J^{s-N+1} u) \, J^{s-2j} u\, dx - c_j' \int_{\mathbb{R}} \psi_1 \, (\partial_x J^{s} u) \, J^{s-2j} u\, dx\\
            &=: A_{4j1} + A_{4j2}.
        \end{aligned}
    \end{equation}
    Since $[\partial_x J^{N-1}, \psi_1] J^{-N+1}$ is a pseudo-differential operator of order $0$, an application of the Cauchy-Schwarz inequality gives us 
    \begin{equation} \label{A4I}
        |A_{4j1}| \lesssim R^2,
    \end{equation}
    for each $j\in \{0,\ldots, N-2\}$. Observe that 
    \begin{equation*}
        A_{402}= -c_0' \int_{\mathbb{R}} \psi_1 \, (\partial_x J^{s} u) \, J^{s} u\, dx = \frac{c_0'}{2} \int_{\mathbb{R}} \psi_1' \,(J^s u)^2 \,dx,
    \end{equation*}
    which implies
    \begin{equation} \label{A4II}
        |A_{402}| \lesssim R^2.
    \end{equation}
    If $0<j\leq N-2$, then
    \begin{equation*}
        \begin{aligned}
            A_{4j2} = c_j' \int_{\mathbb{R}} \psi_1' \, (J^s u) \, J^{s-2j} u \, dx + c_j' \int_{\mathbb{R}} \psi_1 \, J^s u \, (\partial_x J^{s-2j} u) \, dx.
        \end{aligned}
    \end{equation*}
    Since $j>0$, the operator $\partial_x J^{-2j}$ has order $0$, we obtain
    \begin{equation} \label{A4III}
        |A_{4j2}| \lesssim R^2,
    \end{equation} 
    for every $j\in \{1,\ldots, N-2\}$. From \eqref{A4in}-\eqref{A4III}, we arrive at 
    \begin{equation} \label{A4}
        |A_4| \lesssim R^2.
    \end{equation}
    Observe that
    \begin{equation*}
        A_5 = \frac{\alpha}{2} \int_{\mathbb{R}} \psi_1 \, \partial_x (J^{s-N+1} u)^2 \,dx = - \frac{\alpha}{2} \int_{\mathbb{R}} \psi_1' \, (J^{s-N+1} u)^2 \,dx,
    \end{equation*}
    which implies 
    \begin{equation}\label{A5}
        |A_5| \lesssim \int_{\mathbb{R}}  (J^{s-N+1} u)^2 \,dx = \|J^{s-N+1} u\|_{L^2}^2 \leq \|J^s u\|_{L^2}^2 \leq R^2.
    \end{equation}    
For the nonlinear term, when $N=1$, we use integration by parts to write
\begin{equation}\label{A6N1firstdecom}
\begin{aligned}
A_{6k}&=b_k\int_{\mathbb{R}} \psi_1 J^s u\, [J^s,u^k]\partial_x u\, dx+b_k\int_{\mathbb{R}} \psi_1 u^k J^s u \partial_x J^s u\, dx\\
&=b_k\int_{\mathbb{R}} \psi_1 J^s u\, [J^s,u^k]\partial_x u\, dx-\frac{b_k}{2}\int_{\mathbb{R}} \psi_1' u^k  \big(J^s u\big)^2\, dx-\frac{b_k}{2}\int_{\mathbb{R}} \psi_1 \partial_x(u^k)  \big(J^s u\big)^2\, dx\\
&=:A_{6k1}+A_{6k2}+A_{6k3}.
\end{aligned}
\end{equation}  
By Kato-Ponce inequality Lemma \ref{Kato-Ponce}, we deduce
\begin{equation}\label{A6N1firstdecom2}
\begin{aligned}
|A_{6k1}| \lesssim & \|\psi_1\|_{L^{\infty}}\|J^s u\|_{L^2}\|[J^s,u^k]\partial_x u\|_{L^2} \\
\lesssim & \|u\|_{L^{\infty}}^{k-1}\|\partial_x u\|_{L^{\infty}}\|J^s u\|_{L^2}^2+\|\partial_x u\|_{L^{\infty}}\|J^s(u^k)\|_{L^2}\|J^s u\|_{L^2}\\
\lesssim & \|u\|_{L^{\infty}}^{k-1}\|\partial_x u\|_{L^{\infty}}\|J^s u\|_{L^2}^2,
\end{aligned}   
\end{equation}
where we have used Lemma \ref{Grafakos-Oh} to control $\|J^s(u^k)\|_{L^2}$. Since $s>\frac{1}{2}^{+}$, Sobolev embedding $\|u\|_{L^{\infty}}\lesssim \|u\|_{H^s}$ allows us to conclude 
\begin{equation}\label{A6N1}
\begin{aligned}
|A_{6k1}| \lesssim & \|\partial_x u\|_{L^{\infty}}R^{k+1}.
\end{aligned}   
\end{equation}
We get a similar estimate for $A_{6k3}$. We also have $|A_{6k2}|\lesssim R^{k+2}$. In conclusion
\begin{equation*}
 |A_{6k}|\lesssim  R^{k+2}+\|\partial_x u\|_{L^{\infty}}R^{k+1},
\end{equation*}
for all $k\in\{1,\dots, M\}$. Next, we assume that $N\geq 2$. Using that $H^s(\mathbb{R}^d)$ is a Banach algebra and that $s-N+2\leq s$, we deduce 
    \begin{equation} \label{A6N2}
        \begin{aligned}
            |A_{6k}| &= \left|\frac{b_k}{k+1} \int_{\mathbb{R}} \psi_1 \, J^{s-N+1} u \, J^{s-N+1} \partial_x(u^{k+1}) \, dx\right|\\
            &\lesssim \|J^{s-N+1} u\|_{L^2} \|\partial_x J^{-1}(J^{s-N+2} (u^{k+1}))\|_{L^2}\\
            &\lesssim \|J^{s-N+1} u\|_{L^2} \|J^{s-N+2} (u^{k+1})\|_{L^2}\\
            &\lesssim \|J^s u\|^{k+2}\lesssim R^{k+2}.
        \end{aligned}
    \end{equation}    
Gathering \eqref{A1}, \eqref{A2}, \eqref{A3}, \eqref{A4}, \eqref{A5}, \eqref{A6N1} and \eqref{A6N2}, we have for all $\epsilon>0$ that
\begin{equation*}
\frac{1}{2} \frac{d}{dt} \int_{\mathbb{R}} (J^{s-N+1} u)^2 \, \psi_1 \, dx + \left( N + \frac{1}{2} -\epsilon \right)\int_{\mathbb{R}} (J^{s+1} u)^2 \, \psi_1' \, dx \lesssim  \sum_{k=0}^M\big(R+\|\partial_x u\|_{L^{\infty}}\big) R^{k+1},
\end{equation*}
where the implicit constant above depends on $\epsilon>0$. Thus, taking $0<\epsilon<N+\frac{1}{2}$, and integrating the inequality above from $0$ to $T$, we deduce
    \begin{equation*}
        \begin{aligned}
            \frac{1}{2} \int_{\mathbb{R}} (J^{s-N+1} u(\cdot,T))^2 \, \psi_1 \, dx &+ \left( N + \frac{1}{2} -\epsilon \right)\int_0^T \int_{\mathbb{R}} (J^{s+1} u)^2 \, \psi_1' \, dx \, dt \\ &\lesssim \frac{1}{2} \int_{\mathbb{R}} (J^{s-N+1} u_0)^2 \, \psi_1 \, dx + \sum_{k=0}^M\big(RT+\|\partial_x u\|_{L^1_T L^{\infty}_x}\big) R^{k+1}.
        \end{aligned}
    \end{equation*}
Consequently, the above inequality implies that \eqref{Inducteq} holds for $\ell=1$. This is the initial step in the inductive argument. 

Next, we proceed with the inductive step. Let $\ell$ be a positive integer such that $1<\ell\leq N$, and assume that the \eqref{Inducteq} holds for every positive integer $r$ such that $r<\ell$. Our goal is to prove  \eqref{Inducteq} for $\ell$. Consequently, we estimate \eqref{energuestimateSm}.

\underline{\bf Case $1<\ell\leq N$ in \eqref{energuestimateSm}}. We have
    \begin{equation} \label{B1}
        A_1 = \frac{1}{2} \frac{d}{dt} \int_{\mathbb{R}}  \psi_{\ell} \, (J^{s-N+\ell} u)^2 \, dx.
    \end{equation}
    For $A_2$, setting $K:|D|-J$, we get
    \begin{equation*}
    \begin{aligned}
        A_2 &= \gamma \int_{\mathbb{R}} \psi_{\ell} \, J^{s-N+\ell} u \, \partial_x K J^{s-N+\ell} u \, dx + \gamma \int_{\mathbb{R}} \psi_{\ell} \, J^{s-N+\ell} u \, \partial_x J^{s-N+\ell+1} u \, dx\\
            &=:A_{21} + A_{22}.
        \end{aligned}
    \end{equation*}
A similar argument to that in \eqref{A21} shows $|A_{21}| \lesssim R^2$. To estimate $A_{22}$, we have to make a distinction between two cases: $\ell \leq N-1$ and $\ell=N$. If $\ell<N$, we use the same idea from \eqref{A22} to obtain that $|A_{22}| \lesssim R^2$. If $\ell=N$, we write
\begin{equation} \label{A221}
\begin{aligned}
A_{22} &=  \gamma \int_{\mathbb{R}} ([J^{1/2},\psi_{\ell}] J^{s}u) \, \partial_x J^{s+1/2}u \, dx +  \gamma \int_{\mathbb{R}} \psi_{\ell} \, J^{s+1/2}u \, \partial_x J^{s+1/2}u \, dx\\
&= \gamma \int_{\mathbb{R}} ([J^{1/2},\psi_{\ell}] J^{s}u) \, \partial_x J^{s+1/2}u \, dx - \frac{\gamma}{2} \int_{\mathbb{R}} \psi_{\ell}' \, (J^{s+1/2}u)^2 \, dx\\
&=: A_{221} - \frac{\gamma}{2} \int_{\mathbb{R}} \psi_{\ell}' \, (J^{s+1/2}u)^2 \, dx.
\end{aligned}
\end{equation}
By Lemma \ref{LemmaAllpseudo}, the commutator $[J^{1/2},\psi_{\ell}]$ may be decomposed as
    \begin{equation*}
        [J^{1/2},\psi_{\ell}] = -\frac{1}{2} \, \psi_{\ell}' \, \partial_x J^{-3/2} + \Psi_{-3/2}
    \end{equation*}
    for some pseudo-differential operator $\Psi_{-3/2}$ of order $-3/2$. Then, using that $J^2 = 1 - \partial_x^2$,
    \begin{equation*}
        \begin{aligned}
            A_{221} =& - \frac{\gamma}{2} \int_{\mathbb{R}}\psi_{\ell}' \, (\partial_x J^{s-3/2}u) \, \partial_x J^{s+1/2}u \, dx + \gamma \int_{\mathbb{R}}(\Psi_{-3/2} J^{s}u) \, \partial_x J^{s+1/2}u \, dx\\
            =&  \frac{\gamma}{2} \int_{\mathbb{R}}\psi_{\ell}'' \, (\partial_x J^{s-3/2}u) \, J^{s+1/2}u \, dx +  \frac{\gamma}{2} \int_{\mathbb{R}}\psi_{\ell}' \, (\partial_x^2 J^{s-3/2}u) \, J^{s+1/2}u \, dx\\
            &- \gamma \int_{\mathbb{R}}(\partial_x J^{1/2} \Psi_{-3/2} J^{s}u) \, J^{s}u \, dx\\
            =&\frac{\gamma}{2} \int_{\mathbb{R}}\psi_{\ell}'' \, (\partial_x J^{s-3/2}u) \, J^{s+1/2}u \, dx + \frac{\gamma}{2} \int_{\mathbb{R}}\psi_{\ell}' \, ( J^{s-3/2}u) \, J^{s+1/2}u \, dx\\ &- \frac{\gamma}{2} \int_{\mathbb{R}}\psi_{\ell}' \, (J^{s+1/2}u)^2 \, dx - \gamma \int_{\mathbb{R}}(\partial_x J^{1/2} \Psi_{-3/2} J^{s}u) \, J^{s}u \, dx\\
            =:& A_{2211} + A_{2212} - \frac{\gamma}{2} \int_{\mathbb{R}}\psi_{\ell}' \, (J^{s+1/2}u)^2 \, dx + A_{2213}.
        \end{aligned}
    \end{equation*}
By using the commutators $[J^{1/2},\psi_{\ell}'']$ and $[J^{1/2},\psi_{\ell}']$, we may bound the terms $A_{2211}$ and $A_{2212}$ by constant multiple of $R^2$. We omit these calculations to avoid repetitions. Moreover, given that $\partial_x J^{1/2} \Psi_{-3/2}$ is a pseudo-differential operator of order $0$, we may also apply the Cauchy-Schwarz inequality to bound the term $A_{2213}$ by a constant multiple of $R^2$. Summarizing, 
    \begin{equation*}
        A_{22} = - \gamma  \int_{\mathbb{R}}\psi_{\ell}' \, (J^{s+1/2}u)^2 \, dx + \widetilde{A_{22}},
    \end{equation*}
where $|\widetilde{A_{22}}|\lesssim R^2$. Observe that
\begin{equation*}
\begin{aligned}
\int_{\mathbb{R}}\psi_{\ell}' \, (J^{s+1/2}u)^2 \, dx = \int_{\mathbb{R}}([J^{1/2},\psi_{\ell}'] J^{s+1/2}u) \,J^{s}u\, dx + \int_{\mathbb{R}}\psi_{\ell}' \, (J^{s+1}u) \,J^{s}u\, dx.
\end{aligned}
\end{equation*}
On one hand, since $[J^{1/2},\psi_{\ell}']J^{1/2} \in \mathrm{OP}\mathcal{S}^0$, the first term on the right-hand side of the above equation is bounded by a constant multiple of $R^2$. On the other hand, given that $\sqrt{\psi_1'}\Big|_{\supp(\psi_{\ell}')} \equiv 1$, we apply the Cauchy-Schwarz and Young's inequalities to obtain
    \begin{equation*}
        \begin{aligned}
            \left|\int_{\mathbb{R}}\psi_{\ell}' \, (J^{s+1}u) \,J^{s}u\, dx\right| &= \left|\int_{\mathbb{R}}\psi_{\ell}' \sqrt{\psi_1'}\, (J^{s+1}u) \,J^{s}u\, dx\right|\\
            &\lesssim \left\| \sqrt{\psi_1'}\,J^{s+1}u \right\|_{L^2} \|\psi_{\ell}'J^{s}u\|_{L^2}\\
            &\lesssim R^2 + \left\|\sqrt{\psi_1'}\, J^{s+1}u \right\|_{L^2}^2.
        \end{aligned}
    \end{equation*}
Gathering the previous estimates, we arrive at
\begin{equation} \label{B2}
|A_{2}| \lesssim R^2+\left\|  \sqrt{\psi_1'} J^{s+1}u \right\|_{L^2}^2.
\end{equation}
Next, we study $A_3$. We write
\begin{equation*}
\begin{aligned}
A_3 &= - \int_{\mathbb{R}} ([J^N,\psi_{\ell}]J^{s-N+\ell} u) \, \partial_x J^{s+\ell}u \, dx + \frac{1}{2} \int_{\mathbb{R}} \psi_{\ell}' \, (J^{s+\ell} u)^2\, dx\\
&=: A_{31} + \frac{1}{2}\int_{\mathbb{R}}  (J^{s+\ell} u)^2\, \psi_{\ell}' \, dx.
\end{aligned}
\end{equation*}
The idea now is to study the expansion of the commutator $[J^N,\psi_{\ell}]$. By Lemma \ref{LemmaAllpseudo}, we have
\begin{equation*}
\begin{aligned}
\left[J^N,\psi_{\ell}\right] =& - N \psi_{\ell}' \partial_x J^{N-2} - \frac{N(N-1)}{2} \psi_{\ell}'' J^{N-2} + \frac{N(N-2)}{2}\psi_{\ell}'' J^{N-4} \\&+ \sum_{j=3}^{2\ell} c_j \psi_{\ell}^{(j)} \, \Phi_j \, J^{N-j} + K_{N-2\ell-1},
\end{aligned}
\end{equation*}
where each $c_j$ is a real constant, each $\Phi_j \in  \mathrm{OP}\mathcal{S}^0$, $K_{N-2\ell-1}\in  \mathrm{OP}\mathcal{S}^{N-2\ell-1}$, and $\psi_{\ell}^{(j)}$ denotes the $j$-th derivative of the function $\psi_{\ell}$. Thus, 
\begin{equation} \label{B31in}
\begin{aligned}
A_{31} =& N \int_{\mathbb{R}}  \psi_{\ell}' \, (\partial_x J^{s+\ell-2} u) \, \partial_x J^{s+\ell}u\, dx + \frac{N(N-1)}{2}\int_{\mathbb{R}}  \psi_{\ell}'' \, (J^{s+\ell-2} u) \, \partial_x J^{s+\ell}u\, dx\\
&- \frac{N(N-2)}{2}\int_{\mathbb{R}}  \psi_{\ell}'' \, (J^{s+\ell-4} u) \, \partial_x J^{s+\ell}u\, dx\\ &- \sum_{j=3}^{2\ell} c_j\int_{\mathbb{R}}  \psi_{\ell}^{(j)} \, (\Phi_j J^{s+\ell-j} u) \, \partial_x J^{s+\ell}u\, dx - \int_{\mathbb{R}} (K_{N-2\ell-1} J^{s-N+\ell} u) \, \partial_x J^{s+\ell}u\, dx\\
=:&A_{311} + A_{312,1} + A_{312,2} + \sum_{j=3}^{2\ell} A_{31,j} + A_{31,2\ell+1}.
\end{aligned}
\end{equation}
Integration by parts reveals
\begin{equation*}
\begin{aligned}
A_{311}=& - N \int_{\mathbb{R}}  \psi_{\ell}'' \, (\partial_x J^{s+\ell-2} u) \, J^{s+\ell}u\, dx - N \int_{\mathbb{R}}  \psi_{\ell}' \, (J^{s+\ell-2} u) \, J^{s+\ell}u\, dx \\
&+ N \int_{\mathbb{R}}  \psi_{\ell}' \, (J^{s+\ell} u)^2\, dx\\
=:& A_{3111} + A_{3112} + N \int_{\mathbb{R}}   (J^{s+\ell} u)^2\, \psi_{\ell}' \, dx.
\end{aligned}
    \end{equation*}
To estimate $A_{3111}$, we need to expand the commutator $[J,\psi_\ell'']$. We write
\begin{equation*}
\begin{aligned}
\left[J,\psi_{\ell}''\right] = \sum_{r=1}^{2\ell-2} a_r \psi_{\ell}^{(r+2)} \, \Lambda_r J^{1-r} + K_{-2{\ell}+2},
\end{aligned}
\end{equation*}
where each $a_r$ is a real constant, each $\Lambda_r \in \mathrm{OP}\mathcal{S}^0$, and $K_{-2\ell+2}$ is a pseudo-differential operator of order $-2\ell+2$. It follows  
\begin{equation} \label{B3111in}
\begin{aligned}
A_{3111} =& - N \int_{\mathbb{R}} ([J,\psi_{\ell}''] \partial_x J^{s+\ell-2} u) \, J^{s+\ell-1}u\, dx - N \int_{\mathbb{R}}  \psi_{\ell}'' \, (\partial_x J^{s+\ell-1} u) \, J^{s+\ell-1} u\, dx\\
=& - \sum_{r=1}^{2\ell-2} N a_r \int_{\mathbb{R}} \psi_{\ell}^{(r+2)}\, (\Lambda_r J^{1-r} \partial_x J^{-1} J^{s+\ell-1} u) \, J^{s+\ell-1} u\, dx  \\ &- N \int_{\mathbb{R}} (K_{-2\ell+2} \partial_x J^{-1} J^{\ell-1} J^s u) \, J^{s+\ell-1} u \, dx + \frac{N}{2} \int_{\mathbb{R}}  \psi_{\ell}''' \, (J^{s+\ell-1} u)^2 \, dx\\
=& \sum_{r=1}^{2\ell-2} A_{3111,r} + A_{3111,2\ell-1} + A_{3111,2\ell}.
\end{aligned}
\end{equation}
Given that $\psi_{\ell-1}'\Big|_{\supp(\psi_{\ell}^{(r+2)})} \equiv 1$, we get $\psi_{\ell}^{(r+2)}=\psi_{\ell-1}'\psi_{\ell}^{(r+2)}$. Then
\begin{equation} \label{SuppSep}
\begin{aligned}
A_{3111,r} =&  -N a_r \int_{\mathbb{R}} \psi_{\ell}^{(r+2)}\, \psi_{\ell-1}'\,(\Lambda_r J^{1-r} \partial_x J^{-1} \psi_{\ell-1}' J^{s+\ell-1} u) \, J^{s+\ell-1} u\, dx\\
&- N a_r \int_{\mathbb{R}} \psi_{\ell}^{(r+2)}\, \psi_{\ell-1}' \, (\Lambda_r J^{1-r} \partial_x J^{-1} (1-\psi_{\ell-1}') J^{s+\ell-1} u) \, J^{s+\ell-1} u\, dx\\
=:& A_{3111,r}^{(1)} + A_{3111,r}^{(2)}.
\end{aligned}
\end{equation}
Since $\Lambda_r J^{1-r} \partial_x J^{-1}$ has order $0$ for each $r\geq 1$, we obtain
\begin{equation} \label{B3111I}
\begin{aligned}
|A_{3111,r}^{(1)}| &\lesssim \|\Lambda_r J^{1-r} \partial_x J^{-1} (\psi_{\ell-1}' J^{s+\ell-1} u)\|_{L^2} \left\| \sqrt{\psi_{\ell-1}'} \, J^{s+\ell-1} u \right\|_{L^2}\\
&\lesssim \|\psi_{\ell-1}' J^{s+\ell-1} u\|_{L^2} \left\| \sqrt{\psi_{\ell-1}'} \, J^{s+\ell-1} u  \right\|_{L^2}\\
&\lesssim \left\| \sqrt{\psi_{\ell-1}'} \, J^{s+\ell-1} u  \right\|_{L^2}^2.
\end{aligned}
\end{equation}
To estimate the term $A_{3111,r}^{(2)}$, we take an integer $m>1$ be such that $s+\ell-1-2m\leq 0$. Then, we write
\begin{equation*}
 \begin{aligned}
 \psi_{\ell}^{(r+2)} \Lambda_r J^{1-r} \partial_x J^{-1} (1-\psi_{\ell-1}') J^{s+\ell-1} u=& \psi_{\ell}^{(r+2)}\Lambda_r J^{1-r} \partial_x J^{-1} (1-\psi_{\ell-1}')J^{2m} J^{s+\ell-1-2m} u.
 \end{aligned}   
\end{equation*}
Since $J^{2m}$ is a local operator and the supports of $\psi_{\ell}^{(r+2)}$ and $1-\psi_{\ell-1}'$ are separated by a positive constant, we may apply the singular integral representation for pseudo-differential operators and integration by parts to deduce that
\begin{equation*}
\begin{aligned}
\| \psi_{\ell}^{(r+2)}\Lambda_r J^{1-r}  \partial_x J^{-1} (1-\psi_{\ell-1}') J^{2m} J^{s+\ell-1-2m}u\|_{L^2}&\lesssim  \|J^{s+\ell-1-2m}u\|_{L^2}\\
&\lesssim \|u\|_{L^2}\lesssim R.   
\end{aligned}    
\end{equation*}
We conclude from Young's inequality
\begin{equation}\label{B3111r}
\begin{aligned}
| A_{3111,r}^{(2)}| \lesssim R \left\|\sqrt{\psi_{\ell-1}'} \, J^{s+\ell-1} u \,  \right\|_{L^2}\lesssim R^2+\left\|\sqrt{\psi_{\ell-1}'} \, J^{s+\ell-1} u \,  \right\|_{L^2}^2,
\end{aligned}    
\end{equation}
for each $1\leq r\leq 2\ell-2$. 

Next, since $J^{\ell-1} K_{-2\ell+2} \partial_x J^{-1} J^{\ell-1}$ has order $0$, we have 
\begin{equation} \label{B3111III}
\begin{aligned}
|A_{3111,2\ell-1}| \lesssim \|J^{\ell-1} K_{-2\ell+2} \partial_x J^{-1} J^{\ell-1} (J^s u)\|_{L^2} \|J^s u\|_{L^2} \lesssim \|J^s u\|_{L^2}^2 \leq R^2.
\end{aligned}
\end{equation}
Given that $\sqrt{\psi_{\ell-1}'}\Big|_{\supp(\psi_{\ell}''')} \equiv 1$, we get 
\begin{equation} \label{B3111IV}
\begin{aligned}
|A_{3111,2\ell}| \lesssim \left|\int_{\mathbb{R}}  \psi_{\ell}''' \psi_{\ell-1}' \, (J^{s+\ell-1} u)^2 \, dx \right| \lesssim \left\|\sqrt{\psi_{\ell-1}'} \, J^{s+\ell-1} u \right\|_{L^2}^2.
\end{aligned}
\end{equation}
From \eqref{B3111in}, \eqref{B3111I}, \eqref{B3111r}, \eqref{B3111III}, and \eqref{B3111IV}, we conclude that 
    \begin{equation} \label{B3111}
        |A_{3111}| \lesssim R^2+\left\|\sqrt{\psi_{\ell-1}'}\, J^{s+\ell-1} u \right\|_{L^2}^2.
    \end{equation}
The previous argument can also be applied to control $A_{3112}$, obtaining the same bound in \eqref{B3111}. This completes the study of $A_{311}$ in \eqref{B31in}.

To deal with $A_{312,1}$ in \eqref{B31in}, observe that 
\begin{equation*}
\begin{aligned}
A_{312,1} 
=& -\frac{N(N-1)}{2} \int_{\mathbb{R}} ([\partial_x J,\psi_{\ell}''] J^{s+\ell-2} u) \, J^{s+\ell-1} u \, dx \\ &+ \frac{N(N-1)}{4} \int_{\mathbb{R}} \psi_{\ell}'''  (J^{s+\ell-1} u)^2 \, dx\\
=:& A_{312,1}^{(1)} + A_{312,1}^{(2)}.
\end{aligned}
\end{equation*}
By Lemma \ref{LemmaAllpseudo}, we can decompose the commutator in $A_{312,1}^{(1)}$ to use the same ideas in \eqref{SuppSep} to obtain  
\begin{equation*}
|A_{312,1}^{(1)}| \lesssim R^2+\left\| \sqrt{\psi_{\ell-1}'} J^{s+\ell-1} u  \right\|_{L^2}^2.
\end{equation*}
Furthermore, support considerations imply 
    \begin{equation*}
        |A_{312,1}^{(2)}| \lesssim \left\| \sqrt{\psi_{\ell-1}'}\, J^{s+\ell-1} u  \right\|_{L^2}^2.
    \end{equation*}
Therefore,
    \begin{equation} \label{B312I}
        |A_{312,1}| \lesssim R^2+\left\| \sqrt{\psi_{\ell-1}'} J^{s+\ell-1} u  \right\|_{L^2}^2.
    \end{equation}
Expanding the commutators $[\partial_x J,\psi_{\ell}'']$ and $[\partial_x J,\psi_{\ell}^{(j)}]$, and following a similar procedure to that in \eqref{SuppSep}, the same estimate above also holds for the terms $A_{312,2}$ and $A_{31,j}$, $3\leq j\leq 2\ell$ in \eqref{B31in}. We omit these calculations to avoid redundancy.

Since $\partial_x J^{\ell}$ is a skew-adjoint operator, we have
\begin{equation*}
A_{31,2\ell+1} = \int_{\mathbb{R}} (\partial_x J^{\ell} K_{N-2\ell-1} J^{s-N+\ell} u) \, J^s u \, dx,
\end{equation*}
and because $\partial_x J^{\ell} K_{N-2\ell-1} J^{-N+\ell}\in \mathrm{OP}\mathcal{S}^0$, we get
    \begin{equation} \label{B31last}
        |A_{31,2\ell+1}| \lesssim \|J^s u\|_{L^2}^2 \leq R^2.
    \end{equation}
Going back to \eqref{B31in}, from \eqref{B3111}, \eqref{B312I}, \eqref{B31last}, and the above comments, we conclude that
\begin{equation} \label{B3}
A_3 = \left(N+\frac{1}{2}\right) \int_{\mathbb{R}}   (J^{s+\ell} u)^2\, \psi_{\ell}' \, dx + \widetilde{A_3},
\end{equation}
where 
\begin{equation*}
|\widetilde{A_3}| \lesssim  R^2+\left\| \sqrt{\psi_{\ell-1}'}\, J^{s+\ell-1} u\right\|_{L^2}^2. 
\end{equation*}  
Next, we write
\begin{equation*}
\begin{aligned}
A_4 = \sum_{j=0}^{N-2} c_j' \int_{\mathbb{R}} \psi_\ell \, J^{s-N+\ell} u \, (\partial_x J^{s+N+(\ell-2)-2j} u)\, dx=: \sum_{j=0}^{N-2} A_{4,j}.
\end{aligned}
\end{equation*}
For each $j\in \{0,\ldots,N-2\}$, we write
\begin{equation*}
\begin{aligned}
A_{4,j} =& -c_j' \int_{\mathbb{R}} ([\partial_x J^{N-1-j} u, \psi_{\ell}] J^{s-N+\ell} u) \, J^{s+\ell-1-j} u\, dx \\&- c_j' \int_{\mathbb{R}} \psi_{\ell} (\partial_x J^{s+\ell-1-j} u) \, J^{s+\ell-1-j} u\, dx\\
=:& A_{4,j}^{(1)} + A_{4,j}^{(2)}.
\end{aligned}
\end{equation*}
If $j\geq \ell-1$, then both $A_{4,j}^{(1)}$ and $A_{4,j}^{(2)}$ are bounded by a constant $R^2$. On the other hand,  let $j\in \{0,\ldots,\ell-2\}$, by Lemma \ref{LemmaAllpseudo} there exist real constants $q_r$, pseudo-differential operators $\Theta_r$ of order $0$, and a pseudo-differential operator $\Psi_{N-2\ell+1+j}$ of order $N-2\ell+1+j$ such that 
\begin{equation*}
\begin{aligned}
\left[\partial_x J^{N-1-j} u, \psi_{\ell}\right] = \sum_{r=1}^{2\ell-2-2j} q_r \psi_{\ell}^{(r)} \, \Theta_r \, J^{N-j-r} +\Psi_{N-2\ell+1+j}.
\end{aligned}
\end{equation*}
Accordingly, we decompose $A_{4,j}$ as follows
\begin{equation*}
\begin{aligned}
A_{4,j}^{(1)} =& -c_j' \sum_{r=1}^{2\ell-2-2j} q_r \int_{\mathbb{R}} \psi_{\ell}^{(r)} \, (\Theta_r \, J^{s+\ell-j-r} u)\, J^{s+\ell-1-j} u\, dx\\ &- c_j' \int_{\mathbb{R}} (J^{\ell-1-j} \Psi_{N-2\ell+1+j} J^{s-N+\ell} u) \,J^{s} u\, dx.
\end{aligned}
\end{equation*}
Because $J^{\ell-1-j} \Psi_{N-2\ell+1+j} J^{-N+\ell}$ has order $0$, the second term on the right-hand side of the previous equation is bounded by $R^2$. To estimate the remaining factor, we will use the separated support strategy as we did for \eqref{B3111r} above. Thus, each integral in the above sum may be decomposed as 
\begin{equation*}
\begin{aligned}        
&\int_{\mathbb{R}} \psi_{\ell}^{(r)} \psi_{\ell-1-j}' \, (\Theta_r \, J^{1-r} \psi_{\ell-1-j}' J^{s+\ell-1-j} u)\, J^{s+\ell-1-j} u\, dx\\ &+ \int_{\mathbb{R}} \psi_{\ell}^{(r)} \psi_{\ell-1-j}' \, (\Theta_r \, J^{1-r} (1-\psi_{\ell-1-j}') J^{s+\ell-1-j} u)\, J^{s+\ell-1-j} u\, dx.
\end{aligned}        
\end{equation*}
Consequently, by arguing as in \eqref{SuppSep}-\eqref{B3111r}, we obtain that 
\begin{equation*}
|A_{4,j}^{(1)}| \lesssim R^2+ \left\|\sqrt{\psi_{\ell-1-j}'}\, J^{s+\ell-1-j} u \right\|_{L^2}^2.
\end{equation*}
On the other hand, since $\psi'_{\ell}=\psi'_{\ell}\psi'_{\ell-1-j}$, we have
\begin{equation*}
\begin{aligned}
A_{4,j}^{(2)} 
&= \frac{c_j'}{2} \int_{\mathbb{R}} \psi_{\ell}' \, \psi_{\ell-1-j}' \, ( J^{s+\ell-1-j} u)^2 \, dx,
\end{aligned}
\end{equation*}
which implies that 
\begin{equation*}
|A_{4,j}^{(2)}| \lesssim \left\|\sqrt{\psi_{\ell-1-j}'}\, J^{s+\ell-1-j} u \right\|_{L^2}^2.
\end{equation*}
Summarizing, we get
\begin{equation}\label{A4newest}
|A_4| \lesssim R^2+\sum_{j=0}^{\ell-2}  \left\|\sqrt{\psi_{\ell-1-j}'}\, J^{s+\ell-1-j} u \right\|_{L^2}^2 .
    \end{equation}
Recalling that $N\geq 1$ and $\ell\leq N$, it follows
\begin{equation} \label{B6}
|A_5| \lesssim \|J^{s-N+\ell} u\|_{L^2}^2 \leq \|J^s u\|_{L^2}^2 \leq R^2.
\end{equation}
For $A_{6k}$, we have to consider again two cases. If $\ell<N$, we argue as in \eqref{A6N2} to get that $|A_{6k}| \lesssim R^{k+2}$ for every $k\in\{1,\ldots,M\}$. If $\ell=N$, then replacing $\psi_1$ by $\psi_{\ell}$ in  \eqref{A6N1firstdecom}, \eqref{A6N1firstdecom2} and \eqref{A6N1}, we obtain
\begin{equation}\label{B6extra}
\begin{aligned}
|A_{6k}| \lesssim R^{k+2}+\|\partial_x u\|_{L^{\infty}}R^{k+1}.
\end{aligned}
\end{equation}
Gathering together \eqref{B1}, \eqref{B2}, \eqref{B3}, \eqref{A4newest}, \eqref{B6}, and \eqref{B6extra}, we get
\begin{equation*}
\begin{aligned}
\frac{1}{2} \frac{d}{dt} \int_{\mathbb{R}}  \psi_{\ell} \, &(J^{s-N+\ell} u)^2 \, dx + \left(N+\frac{1}{2}\right) \int_{\mathbb{R}}   (J^{s+\ell} u)^2\, \psi_\ell' \, dx\\
\lesssim&  \sum_{j=0}^{\ell-2}  \left\|\sqrt{\psi_{\ell-1-j}'}\, J^{s+\ell-1-j} u  \right\|_{L^2}^2+ \sum_{k=0}^{M} \big(R + \|\partial_x u\|_{L^{\infty}}\big)R^{k+1}.
\end{aligned}
\end{equation*}
Integrating both sides of the above expression from $0$ to $T$ yields
\begin{equation*}
\begin{aligned}
\frac{1}{2} \int_{\mathbb{R}}  \psi_{\ell} \, (J^{s-N+\ell} u(\cdot,T))^2 \, dx &+ \left(N+\frac{1}{2}\right) \int_0^T \int_{\mathbb{R}}   (J^{s+\ell} u)^2\, \psi_{\ell}' \, dx \, dt\\
\lesssim&  \int_{\mathbb{R}} \psi_{\ell }(J^{s-N+\ell} u_0)^2  \, dx + \sum_{j=0}^{\ell-2} \int_0^T \int_{\mathbb{R}}  \psi_{\ell-1-j}'\, (J^{s+\ell-1-j} u)^2 \, dx \, dt \\ 
& + \sum_{k=0}^{M} \big(RT + \|\partial_x u\|_{L^1_TL^{\infty}}\big)R^{k+1}.
\end{aligned}
\end{equation*}
All of the terms on the right-hand side above are finite due to the induction hypothesis \eqref{Inducteq} and from our LWP results Lemmas \ref{LWPN12} and \ref{StandardLWP}. As a consequence, we arrive at
\begin{equation*}
\int_0^T \int_{\mathbb{R}} (J^{s+\ell} u)^2 \, \psi_\ell' \, dx \, dt \leq c(\|u_0\|_{H^s};\|\partial_xu\|_{L_T^1 L_x^\infty};T),
\end{equation*}
which establishes the inductive step. Thus, we conclude that \eqref{Inducteq} is valid for all $\ell\leq N$.

Let $R>0$ and $\widetilde{k} \in (0,N)$. Let us assume that $A = 2^{N+1}R$ (this constant $A$ is the one used for the construction of the sequence $(\psi_{\ell
})_{\ell\geq 1}$). By the previous result, we know that 
\begin{equation*}
\psi_N' \,  J^{s+N}u \in L^2(\mathbb{R} \times (0,T)).    
\end{equation*}
Let $\phi_R$ be a smooth function defined on $\mathbb{R}$ such that $0\leq \phi_R\leq 1$ and $\supp(\phi_R) \subseteq [-R,R]$. Since $\psi_N' \equiv 1$ in $[-2^{-N}A,2^{-N}A] = [-2R,2R]$, it follows that 
\begin{equation*}
\dist(\supp(1-\psi_N'),\supp(\phi_R)))\geq \delta    
\end{equation*}
for some positive constant $\delta$. By Lemma \ref{LocReg} (iii), we get $$\phi_R J^{s+\widetilde{k}}u \in L^2(\mathbb{R} \times (0,T)),$$ which proves the desired result \eqref{Katosmoothingeq} for the operators $A^r=J^r$ with $r\in [0,s+N]$.

Finally, by applying several times Lemma \ref{LocReg} (ii), we can control $\|\phi_R \, |D|^r u\|_{L^2(\mathbb{R} \times (0,T))}$ and $\|\phi_R \, \partial_x^{\lfloor r \rfloor}|D|^{\{r\}} u\|_{L^2(\mathbb{R} \times (0,T))}$ for any $r \in [0,s+N]$. The proof is complete.
\end{proof}


\section{Propagation of regularity: Proof of Theorem \ref{ThrmPropagation}}\label{propSection}

In this part, we show the propagation of regularity principle for solutions of \eqref{Benjamineq}. The main ingredients are the results from pseudo-differential calculus and Kato's smoothing effect established in Theorem \ref{SmoothingEffectThrm}. For the deduction of this principle and similar strategies for other dispersive models, we refer to  \cite{FreireMendezRiano2022,IsazaLinaresPonce2015,IsazaLinaresPonceBO2016,KenigLinaresPonceVega2018}.

To prove Theorem \ref{ThrmPropagation}, we begin by introducing certain weighted functions. Next, we establish general estimates that will be useful in obtaining some key propositions. These, in turn, will ultimately lead to the proof of the propagation of regularity principle for solutions of \eqref{Benjamineq}.

\subsection{Weighted functions} To carry out our energy estimates, we use the construction of weighted functions introduced in \cite{IsazaLinaresPonce2015}.

\begin{lemma}
    There exists a family of functions $\{\chi_{\varepsilon,b} : \varepsilon>0, b\geq 5\varepsilon\}$ satisfying the following properties:
    \begin{enumerate}
        \item $\chi_{\varepsilon,b}'\geq 0$,
        \item $\chi_{\varepsilon,b} = \begin{cases}
            0, &\text{ if } x\leq \varepsilon,\\
            1, &\text{ if } x\geq b,
        \end{cases}$
        \item $\chi_{\varepsilon,b}'\geq \frac{1}{b-3\varepsilon} \mathds{1}_{[3\varepsilon,b-2\varepsilon]}$,
        \item $\chi_{\varepsilon,b} \geq \frac{1}{2} \frac{\varepsilon}{b-3\varepsilon} \mathds{1}_{[3\varepsilon,\infty)}$,
        \item $\supp(\chi_{\varepsilon,b}') \subseteq [\varepsilon,b]$,
    \end{enumerate}
where $ \mathds{1}_A$ stands for the indicator function over a set $A$.
\end{lemma}

To estimate the nonlinear terms in \eqref{Benjamineq}, we consider some suitable partitions of unity. Consider $\epsilon>0$, $b\geq 5\epsilon$. Let $\psi_{\varepsilon}$ be a smooth function such that 
\begin{equation*}
0\leq \psi_{\varepsilon} \leq 1 \quad\text{ and } \quad \supp(\psi_{\varepsilon}) \subseteq (-\infty,\varepsilon/2].    
\end{equation*}
Let $\phi_{\varepsilon,b}$ be a smooth function with compact support such that 
\begin{equation*}
\phi_{\varepsilon,b} = 1 \text{ in } [\varepsilon/2,\varepsilon], \quad  \supp(\phi_{\varepsilon,b}) \subseteq [\varepsilon/4,b],  
\end{equation*}
and 
\begin{equation}\label{partitionkunity0}
\chi_{\varepsilon,b}(x) + \phi_{\varepsilon,b}(x) + \psi_\varepsilon(x) = 1,    
\end{equation}
for each $x\in\mathbb{R}$. Let $k \in \{2,\dots,M+1\}$ (recall that $M$ determines the number and powers of nonlinear terms), let $\widetilde{\phi_{\varepsilon,b,k}} \in \mathcal{C}_0^{\infty}(\mathbb{R})$ be such that 
\begin{equation*}
\widetilde{\phi_{\varepsilon,b,k}} = 1 \text{ in } [\varepsilon/2,\varepsilon], \quad  \supp(\widetilde{\phi_{\varepsilon,b,k}}) \subseteq [\varepsilon/4,b]    
\end{equation*}
and
\begin{equation}\label{partitionkunity}
\chi_{\varepsilon,b}^{k}(x) + \widetilde{\phi_{\varepsilon,b,k}}^{k}(x) + \psi_\varepsilon(x) = 1,    
\end{equation}
for all $x\in\mathbb{R}$. Here $\chi_{\varepsilon,b}^{k}=\big(\chi_{\varepsilon,b}\big)^{k}$ and $ \widetilde{\phi_{\varepsilon,b,k}}^{k}= \big(\widetilde{\phi_{\varepsilon,b,k}}\big)^{k}$.

Finally, let $\eta_{\varepsilon,b}$ be a smooth function such that $0\leq \eta_{\varepsilon,b}\leq 1$ with
\begin{equation*}
\eta_{\varepsilon,b} = 1 \text{ in } [\varepsilon/7,b+3\varepsilon/4] \quad \text{and} \quad  \supp(\eta_{\varepsilon,b}) \subseteq [\varepsilon/8,b+\varepsilon].    
\end{equation*}
Let $\theta_{\varepsilon,b}$ be a smooth function such that $0\leq \theta_{\varepsilon,b}\leq 1$ with
\begin{equation}\label{thetafunctsupp}
\theta_{\varepsilon,b} = 1 \text{ in } [\varepsilon/5,b+\varepsilon/4] \quad \text{and} \quad  \supp(\theta_{\varepsilon,b}) \subseteq [\varepsilon/6,b+\varepsilon/2].    
\end{equation}

\subsection{Preliminary results}

As we emphasized at the beginning of the proof of Theorem \ref{SmoothingEffectThrm}, using the continuous dependence and an approximation argument by smooth solutions, we will assume that $u\in C([0,T];H^{\infty}(\mathbb{R}))$. In the end, the desired result for $u \in C([0,T];H^s(\mathbb{R}))$ solution of \eqref{Benjamineq} with initial condition $u_0\in H^s(\mathbb{R}) \cap H^m ((x_0,\infty))$ with $m>s>\min\{\frac{3}{2},\frac{2N+1}{4}\}$ follows by taking the limit to our estimates. Additionally, we work with the equation in \eqref{Benj2}, which is equivalent to that in \eqref{Benjamineq}.

By translation, we may assume that $x_0 = 0$. For brevity in the notation, we assume that $\chi_{\varepsilon,b}$, $\theta_{\varepsilon,b}$, and $\eta_{\varepsilon,b}$ are acting as $\chi_{\varepsilon,b}(x+vt)$, $\theta_{\varepsilon,b}(x+vt)$, and $\eta_{\varepsilon,b}(x+vt)$. 

Applying the operator $J^m$ to the equation \eqref{Benj2}, multiplying by $\chi_{\varepsilon,b}^2 J^m u$, and integrating with respect to the $x$-variable, after integration by parts, we arrive at
\begin{equation} \label{EqProp2}
\begin{aligned}
\frac{1}{2} \frac{d}{dt} &\int_{\mathbb{R}} \chi_{\varepsilon,b}^2 \,   (J^{m} u)^2 \, dx  \underbrace{- \left(\frac{v}{2} + \frac{\alpha}{2} \right) \int_{\mathbb{R}} (\chi_{\varepsilon,b}^2)' \,  (J^{m} u)^2 \, dx}_{A_1}\\ &+ \underbrace{\gamma \int_{\mathbb{R}} \chi_{\varepsilon,b}^2 \, J^{m} u \, \partial_x |D| J^{m} u \, dx}_{A_2} \underbrace{- \int_{\mathbb{R}}\chi_{\varepsilon,b}^2 \, J^{m} u \, \partial_x J^{2N+m} u\, dx}_{A_3}\\  
&+\underbrace{\int_{\mathbb{R}} \chi_{\varepsilon,b}^2 \, J^{m} u \, Q(J) J^{m} u\, dx}_{A_4}+\underbrace{\sum_{k=1}^M b_{k}\int_{\mathbb{R}} \chi_{\varepsilon,b}^2 \, J^{m} u \, J^{m} (u^{k}\partial_x u)\, dx}_{A_5} = 0.
\end{aligned}
\end{equation}
In what follows, we estimate each one of the terms $A_j=A_j(t)$, $j=1,\dots,5$. Since these calculations are independent of the value of $m>s$, we present each of them as a proposition. 

\begin{proposition}\label{PropagA1}
Let $m>s>\min\{\frac{3}{2},\frac{2N+1}{4}\}$ and $u\in C([0,T];H^{\infty}(\mathbb{R}))$. It follows
    \begin{equation*}
        |A_1(t)| \lesssim \|\eta_{\varepsilon,b} \, J^{m+N-1} u(t)\|_{L^2}^2 + \|u\|_{L_T^{\infty} H_x^s}^2.
    \end{equation*}
\end{proposition}
\begin{proof}
Our choice of the function $\theta_{\varepsilon,b}$ (see, \eqref{thetafunctsupp}) implies that $|\chi_{\varepsilon,b}'|\lesssim \theta_{\varepsilon,b}$, and $ \theta_{\varepsilon,b}$ and $ \eta_{\varepsilon,b}$ satisfy the hypothesis of Lemma \ref{LocReg}. Consequently, Lemma \ref{LocReg} (iii) imply
    \begin{equation*}
        \begin{aligned}
            |A_1(t)| \lesssim \|\theta_{\varepsilon,b} \, J^m u\|_{L^2}^2 \lesssim \|\eta_{\varepsilon,b} \, J^{m+N-1} u\|_{L^2}^2 + \|u\|_{L^2}^2.
        \end{aligned}
    \end{equation*}
\end{proof}

\begin{proposition}\label{PropagA3}
Let $m>s>\min\{\frac{3}{2},\frac{2N+1}{4}\}$ and $u\in C([0,T];H^{\infty}(\mathbb{R}))$. Then there exists a positive constant $C_N\in \left( 0 , N+\frac{1}{2} \right)$ such that
\begin{equation*}
A_3(t) = C_N \int_{\mathbb{R}} (\chi_{\varepsilon,b}^2)' \, (J^{N+m}u)^2 \, dx  + R_{A_3}(t),    
\end{equation*}    
where
\begin{equation*}
|R_{A_3}(t)| \lesssim \\ \|\eta_{\varepsilon,b} \, J^{m+N-1} u(t)\|_{L^2}^2+ \|u\|_{L_T^{\infty} H_x^s}^2.    
\end{equation*}
\end{proposition}

\begin{proof}
By using commutators and integration by parts, we find
\begin{equation*}
\begin{aligned}
A_3 &= - \int_{\mathbb{R}}([J^N,\chi_{\varepsilon,b}^2] \, J^{m} u) \, \partial_x J^{N+m} u\, dx + \frac{1}{2} \int_{\mathbb{R}}(\chi_{\varepsilon,b}^2)' \, (J^{N+m} u )^2 \, dx\\ 
&=: A_{31} + \frac{1}{2} \int_{\mathbb{R}}(\chi_{\varepsilon,b}^2)' \, (J^{N+m} u )^2 \, dx.
\end{aligned}
\end{equation*}
Let $R$ be a positive integer such that $R\geq 2(N+m)$. By Lemma \ref{LemmaAllpseudo}, we expand the commutator above as follows
\begin{equation*}
\begin{aligned}
\left[J^N,\chi_{\varepsilon,b}^2\right] =& -N(\chi_{\varepsilon,b}^2)' \partial_x J^{N-2} - \frac{N(N-1)}{2} (\chi_{\varepsilon,b}^2)'' J^{N-2} \\ &+ \frac{N(N-2)}{2}(\chi_{\varepsilon,b}^2)'' J^{N-4}+\sum_{k=3}^R d_k (\chi_{\varepsilon,b}^2)^{(k)} \Phi^k J^{N-k}\\ &+ \Phi_{N-R-1},
\end{aligned}
\end{equation*}
for some constants $d_k$, some pseudo-differential operators $\Phi^k \in \mathrm{OP}\mathcal{S}^{0}$, and a pseudo-differential operator $\Phi_{N-R-1}\in \mathrm{OP}\mathcal{S}^{N-R-1}$. Then, $A_{31}$ may be rewritten as follows
\begin{equation}\label{A31prop}
    \begin{aligned}
        A_{31} =& N \int_{\mathbb{R}} (\chi_{\varepsilon,b}^2)' \, (\partial_x J^{N-2} J^m u) \, \partial_x J^{N+m} u \, dx + \frac{N(N-1)}{2} \int_{\mathbb{R}} (\chi_{\varepsilon,b}^2)'' \, (J^{N-2} J^m u) \, \partial_x J^{N+m} u \, dx \\ &-\frac{N(N-2)}{2} \int_{\mathbb{R}} (\chi_{\varepsilon,b}^2)'' \, (J^{N-4} J^m u) \, \partial_x J^{N+m} u \, dx \\&- \sum_{k=3}^R d_k \int_{\mathbb{R}}(\chi_{\varepsilon,b}^2)^{(k)} \, (\Phi^k J^{N-k}J^mu) \, \partial_x J^{N+m} u \, dx - \int_{\mathbb{R}} (\Phi_{N-R-1} J^m u) \, \partial_x J^{N+m} u \, dx\\
        =:& A_{31}^{(1)} + A_{31}^{(2,1)} + A_{31}^{(2,2)} + \sum_{k=3}^{R} A_{31}^{(k)} + A_{31}^{(R+1)}.
    \end{aligned}
\end{equation}
Integration by parts and writing $\partial_x^2 = 1-J^2 $ yield
\begin{equation}\label{A31(11)}
\begin{aligned}
    A_{31}^{(1)} =& -  N \int_{\mathbb{R}} (\chi_{\varepsilon,b}^2)'' \, (\partial_x J^{N-2+m}u) \, J^{N+m} u \, dx -  N \int_{\mathbb{R}} (\chi_{\varepsilon,b}^2)' \, (\partial_x^2 J^{N-2+m}u) \, J^{N+m} u \, dx\\ =& -  N \int_{\mathbb{R}} (\chi_{\varepsilon,b}^2)'' \, (\partial_x J^{N-2+m}u) \, J^{N+m} u \, dx+ N \int_{\mathbb{R}} (\chi_{\varepsilon,b}^2)' \, ( J^{N+m}u) ^2\,dx \\ &-  N \int_{\mathbb{R}} (\chi_{\varepsilon,b}^2)' \, (J^{N-2+m}u) \, J^{N+m} u \, dx \\ =& A_{31}^{(1,1)} + N \int_{\mathbb{R}} (\chi_{\varepsilon,b}^2)' \, ( J^{N+m}u) ^2\,dx + A_{31}^{(1,2)}.
\end{aligned}
\end{equation}
Notice that 
\begin{equation*}
\begin{aligned}
A_{31}^{(1,1)} =& - N \int_{\mathbb{R}} ([J,(\chi_{\varepsilon,b}^2)''] \, \partial_x J^{N-2+m}u) \, J^{N+m-1} u \, dx \\ &+ \frac{N}{2} \int_{\mathbb{R}} (\chi_{\varepsilon,b}^2)'''(J^{N+m-1} u)^2 \, dx\\ =:&  A_{31}^{(1,1,1)}  + A_{31}^{(1,1,2)}.
\end{aligned}
\end{equation*}
Since $|(\chi_{\varepsilon,b}^2)'''| \lesssim  \eta_{\varepsilon,b}^2$, we get 
\begin{equation*}
    |A_{31}^{(1,1,2)}| \lesssim \int_{\mathbb{R}} \eta_{\varepsilon,b}^2\, (J^{N+m-1} u)^2\, dx.
\end{equation*}
On the other hand, let $\widetilde{m}$ be an integer such that $\widetilde{m} \geq 2(N+m-1)$. By Lemma \ref{LemmaAllpseudo}, we have that for each $j\in\{1,\dots,\widetilde{m}\}$ there exist a constant $c_j$ and $\Psi^{-j}\in \mathrm{OP}\mathcal{S}^{-(j-1)}$ such that  
\begin{equation*}
[J,(\chi_{\varepsilon,b}^2)''] = \sum_{j=1}^{\widetilde{m}} c_j (\chi_{\varepsilon,b}^2)^{(j+2)} \Psi^{-j} + K_{-\widetilde{m}},    
\end{equation*}
where $K_{-\widetilde{m}}\in \mathrm{OP}\mathcal{S}^{-\widetilde{m}}$. Applying this decomposition, we get
\begin{equation*}
    \begin{aligned}
         A_{31}^{(1,1,1)} =& -\sum_{j=1}^{\widetilde{m}} N c_j \int_{\mathbb{R}}(\chi_{\varepsilon,b}^2)^{(j+2)} \, (\Psi^{-j}  \partial_x J^{N-2+m}u) \, J^{N+m-1} u \, dx\\ &-N \int_{\mathbb{R}} (K_{-\widetilde{m} }\partial_x J^{N-2+m}u) \,  J^{N+m-1} u \, dx\\ =:& \sum_{j=1}^{\widetilde{m}} A_{31}^{(1,1,1,j)} +A_{31}^{(1,1,1,\widetilde{m}+1)}. 
    \end{aligned}
\end{equation*}
Since $\eta_{\varepsilon,b} \equiv 1$ in $\supp((\chi_{\varepsilon,b}^2)')$, it follows that $\eta_{\varepsilon,b}(\chi_{\varepsilon,b}^2)^{(j)}  = (\chi_{\varepsilon,b}^2)^{(j)}$ for each positive integer $j$. Additionally, there is a positive constant $\delta$ such that $$\dist(\supp((\chi_{\varepsilon,b}^2)^{(j)}) , \supp(1-\eta_{\varepsilon,b})) > \delta$$ for every positive integer $j$. Hence, by using the separated support strategy as in \eqref{SuppSep}-\eqref{B3111r} (which depends on Lemma \ref{IneqSupports}), we deduce 
\begin{equation} \label{SepSuppProp}
    \begin{aligned}
|A_{31}^{(1,1,1,j)}|  \lesssim& \left| \int_{\mathbb{R}}(\chi_{\varepsilon,b}^2)^{(j+2)}  \eta_{\varepsilon,b} \, (\Psi^{-j} \partial_x J^{-1} (\eta_{\varepsilon,b} J^{N+m-1}u)) \, J^{N+m-1} u \, dx\right|\\ &+ \left| \int_{\mathbb{R}}(\chi_{\varepsilon,b}^2)^{(j+2)}  \eta_{\varepsilon,b} \, (\Psi^{-j} \partial_x J^{-1} ((1-\eta_{\varepsilon,b}) J^{N-1+m}u)) \, J^{N+m-1} u \, dx\right|\\ \lesssim& \left(\|\eta_{\varepsilon,b} J^{N+m-1}u\|_{L^2} + \|(\chi_{\varepsilon,b}^2)^{(j+2)}  \, (\Psi^{-j} \partial_x J^{-1} ((1-\eta_{\varepsilon,b}) J^{N+m-1}u))\|_{L^2}\right)\\
&\times \|\eta_{\varepsilon,b} J^{N+m-1} u\|_{L^2}\\
\lesssim& \left(\|\eta_{\varepsilon,b} J^{N+m-1}u\|_{L^2} +\|u\|_{L^{\infty}_TH^s_x}\right) \|\eta_{\varepsilon,b} J^{N+m-1} u\|_{L^2}.
\end{aligned}
\end{equation}
In addition, our choice of $\widetilde{m}$ and the continuity of pseudo-differential operators establish $|A_{31}^{(1,1,1,\widetilde{m}+1)}| \lesssim \|u\|_{L^2}^2$.
   
Let $\zeta =\sqrt{(\chi_{\varepsilon,b}^2)'}$. Given $\omega>0$ small enough, an application of the Cauchy-Schwarz and Young's inequality with $\omega>0$, together with Lemma \ref{LocReg} (ii) yields
\begin{equation} \label{CSomega}
    \begin{aligned}
        |A_{31}^{(1,2)}| &\leq N \|\zeta \, J^{N+m-2}u\|_{L^2} \|\zeta\, J^{N+m} u\|_{L^2}\\ &\leq  C_\omega\|\zeta \, J^{N+m-2}u\|_{L^2}^2 + \omega\|\zeta\, J^{N+m} u\|_{L^2}^2\\ &\leq C_\omega' \|\theta_{\varepsilon,b} \, J^{N+m-2}u\|_{L^2}^2 + \omega\|\zeta\, J^{N+m} u\|_{L^2}^2\\ &\leq C_\omega'' \|\eta_{\varepsilon,b} \, J^{N+m-1}u\|_{L^2}^2 + \omega\|\zeta\, J^{N+m} u\|_{L^2}^2.
    \end{aligned}
\end{equation}
This completes the estimate of $A_{31}^{(1)}$ in \eqref{A31(11)}.  Now,   using Lemma \ref{LemmaAllpseudo}, we expand the commutator $[\partial_x J,(\chi^2_{\varepsilon,b})'']$ to get
\begin{equation*}
\begin{aligned}
A_{31}^{(2,1)}  =& -\frac{N(N-1)}{2} \sum_{r=1}^{\widetilde{R}} \lambda_r \int_{\mathbb{R}} (\chi_{\varepsilon,b}^2)^{(r+2)} \, (\Lambda^r J^{1-r} J^{N+m-1}u) \, J^{N+m-1} u \, dx \\ 
&- \frac{N(N-1)}{2} \int_{\mathbb{R}} (\Lambda^{\widetilde{R}+1}J^{N+m-2}u) \, J^{N+m-1} u \, dx \\ &+\frac{N(N-1)}{4} \int_{\mathbb{R}} (\chi_{\varepsilon,b}^2)''' \, (J^{N+m-1} u)^2  \, dx\\ 
=:& \sum_{r=1}^{\widetilde{R}} A_{31}^{(2,1,r)} + A_{31}^{(2,1,\widetilde{R}+1)} +A_{31}^{(2,1,\widetilde{R}+2)},
\end{aligned}
\end{equation*}
where $\widetilde{R}$ is a positive integer such that $\widetilde{R}\geq 2(N+m-1)$, the $\lambda_r$'s are constants, $\Lambda^r \in \mathrm{OP}\mathcal{S}^{0}$, and $\Lambda^{\widetilde{R}+1}\in \mathrm{OP}\mathcal{S}^{1-\widetilde{R}}$.

Using the strategy of separated supports as in \eqref{SepSuppProp}, we deduce that for each $r\in \{1,\dots,\widetilde{R}\}$,
\begin{equation*}
\begin{aligned}
|A_{31}^{(2,1,r)}| \lesssim& \left(\|\eta_{\varepsilon,b} J^{N-1+m}u\|_{L^2} +\|u\|_{L^{\infty}_TH^s_x}\right) \|\eta_{\varepsilon,b} J^{N+m-1} u\|_{L^2}.
 \end{aligned}       
\end{equation*}
Next, from the choice of $\widetilde{R}$, $A_{31}^{(2,1,\widetilde{R}+1)}$ is bounded by a constant multiple of $\|u\|_{L^2}^2$. Again, considering the support of the involved functions, the same estimate in \eqref{SepSuppProp} holds for $A_{31}^{(2,1,\widetilde{R}+2)}$. This concludes the estimates for $A_{31}^{(2,1)}$. Building on the previous ideas, we get similar bounds for $A_{31}^{(2,2)}$ and $A_{31}^{(k)}$ for every $k\in \{3,\ldots,R+1\}$. 

Finally, collecting all the previous estimates, we have deduced that for all $\omega>0$,
\begin{equation*}
A_3(t) = \Big(N+\frac{1}{2}-\omega\Big) \int_{\mathbb{R}} (\chi_{\varepsilon,b}^2)' \, (J^{N+m}u)^2 \, dx  + R_{A_3}(t),    
\end{equation*}

This finishes the estimate of $A_3(t)$. 
\end{proof}

\begin{proposition}\label{PropagA4}
Let $m>s>\min\{\frac{3}{2},\frac{2N+1}{4}\}$ and $u\in C([0,T];H^{\infty}(\mathbb{R}))$. It follows
    \begin{equation*}
        |A_4(t)| \lesssim \|\eta_{\varepsilon,b} J^{m+N-1} u(t)\|_{L^2}^2  + \|u\|_{L_T^{\infty} H_x^s}^2.
    \end{equation*}
\end{proposition}
\begin{proof}
The term $A_4$ consists of lower-order dispersions, and its estimate follows a reasoning similar to that in the proof of Proposition \ref{PropagA3}, which we remark is based on commutator expansions and Lemma \ref{LocReg} (iii). Since this type of argument has already been presented in sufficient detail in the proof of Theorem \ref{SmoothingEffectThrm} and Proposition \ref{PropagA3}, we omit the proof here to avoid repetition.
\end{proof}

\begin{proposition} \label{PropagA2}
Let $m>s>\min\{\frac{3}{2},\frac{2N+1}{4}\}$ and $u\in C([0,T];H^{\infty}(\mathbb{R}))$. If $N=1$, then for any $\omega>0$ there exists a positive constant $C_\omega$ such that
\begin{equation*}
|A_2(t)| \leq \omega \int_{\mathbb{R}} (\chi_{\varepsilon,b}^2)' \, (J^{m+1} u)^2 \, dx + C_{\omega} \Big(\|\chi_{\varepsilon,b} \, J^m u(t)\|_{L^2}^2 + \|\eta_{\varepsilon,b} \, J^{m} u(t)\|_{L^2}^2 + \|u(t)\|_{L^2}^2\Big).
\end{equation*}
Otherwise, $A_2(t)$ satisfies 
\begin{equation*}
|A_2(t)| \lesssim \|\chi_{\varepsilon,b} \, J^m u(t)\|_{L^2}^2 + \|\eta_{\varepsilon,b} \, J^{m+N-1} u(t)\|_{L^2}^2 + \|u(t)\|_{L^2}^2.
    \end{equation*}
\end{proposition}
\begin{proof}
To deduce this proposition, we perform a more detailed decomposition of the operator $|D|-J$. From Proposition \ref{ExpDs-Js}, we can rewrite $A_2$ as follows 
\begin{equation*}
\begin{aligned}
A_2 =& \gamma \int_{\mathbb{R}} \chi_{\varepsilon,b}^2 \, J^m u \, \partial_x (|D|-J)J^m u\, dx + \gamma \int_{\mathbb{R}} \chi_{\varepsilon,b}^2 \, J^m u \, \partial_x J^{m+1} u\, dx\\ =& \gamma \sum_{j=1}^{\widetilde{L}-1}  (-1)^j\binom{1/2}{j} \int_{\mathbb{R}} \chi_{\varepsilon,b}^2 \, J^m u \, \partial_x J^{-(2j-1)} J^m u\, dx + \gamma \int_{\mathbb{R}} \chi_{\varepsilon,b}^2 \, J^m u \, \partial_x \Phi J^m u\, dx\\ &+\gamma \int_{\mathbb{R}} \chi_{\varepsilon,b}^2 \, J^m u \, \partial_x J^{m+1} u\, dx\\ =:&  \sum_{j=1}^{\widetilde{L}-1} A_2^{(j)} + A_2^{(\widetilde{L}+1)} + A_{21},
    \end{aligned}
\end{equation*}
where $\widetilde{L}$ is a positive integer such that $2\widetilde{L} \geq m+2$, and by \eqref{conmexpansionJD} and \eqref{estimateJDGen}, $\|\Phi f\|_{L^2}\lesssim \|f\|_{H^{-m-1}}$ with $[\partial_x,\Phi]=0$. Consequently, applying these facts and the Cauchy-Schwarz and Young inequalities, we get
\begin{equation*}
\begin{aligned}
|A_2^{(\widetilde{L}+1)}| &\lesssim \|\chi_{\varepsilon,b} J^m u\|_{L^2} \|\partial_x\Phi J^m u\|_{L^2} \\
&\lesssim \|\chi_{\varepsilon,b} J^m u\|_{L^2}^2 + \|u\|_{L^2}^2.
\end{aligned}
\end{equation*}
Let $j\in \{1,\dots,\widetilde{L}-1\}$ be fixed. Then, 
\begin{equation*}
\begin{aligned}
A_{2}^{(j)} =& - \gamma (-1)^j\binom{1/2}{j} \int_{\mathbb{R}} ([\partial_x J^{-j+1/2},\chi_{\varepsilon,b}^2] J^m u) \, J^{-j+1/2+m}u \, dx\\ &+ \frac{\gamma}{2} (-1)^j\binom{1/2}{j} \int_{\mathbb{R}} (\chi_{\varepsilon,b}^2)'\, ( J^{-j+1/2+m}u)^2 \, dx\\ =:& A_2^{(j,1)} + A_2^{(j,2)}.
\end{aligned}
\end{equation*}
Applying Lemma \ref{LocReg} (iii), we get
\begin{equation*}
\begin{aligned}
|A_2^{(j,2)}|\lesssim\|\theta_{\varepsilon,b} \, J^{-j+1/2+m}u \|_{L^2}^2\lesssim \|\eta_{\varepsilon,b} \, J^{m+N-1}\|_{L^2}^2 + \|u\|_{L^2}^2.
    \end{aligned}
\end{equation*}
To estimate $A_2^{(j,1)}$, we may decompose the commutator $[\partial_x J^{-j+1/2},\chi_{\varepsilon,b}^2]$ to apply arguments previously used, e.g, see estimate for $A_{31}^{(1)}$ in the proof of  Proposition \ref{PropagA3}. Finally, integration by parts and commutator estimates yield
\begin{equation*}
\begin{aligned}
A_{21} =& -\frac{\gamma}{2} \int_{\mathbb{R}} (\chi_{\varepsilon,b}^2)' \, J^m u \, J^{m+1}u \, dx - \frac{\gamma}{2} \int_{\mathbb{R}} ([J,\chi_{\varepsilon,b}^2] \, \partial_x J^m u) \, J^{m}u \, dx.
\end{aligned}
\end{equation*}
Once again, expanding the above commutator and familiar arguments leads us to the desired estimate. It should be noted that, if $N=1$, the first term on the right-hand side of the expression above requires applying the argument in \eqref{CSomega}. The reason is that the factor $\|\sqrt{(\chi_{\varepsilon,b}^2)'} J^{m+1} u\|_{L^2}^2$ would not be controlled by Lemma \ref{LocReg} (iii) (this is precisely the smoothing effect we aim to establish at this stage when $N=1$). The proof is complete.
\end{proof} 

\begin{proposition} \label{PropagA5}
Let $m>s>\min\{\frac{3}{2},\frac{2N+1}{4}\}$ and $u\in C([0,T];H^{\infty}(\mathbb{R}))$. It follows
\begin{equation*}
\begin{aligned}
|A_5(t)| \lesssim &\big(1+\|u\|_{L^{\infty}_TH^{\frac{1}{2}^{+}}}\big)^{M+2} + \big(1+\|u\|_{L^{\infty}_TH^{\frac{1}{2}^{+}}}\big)^{M+1}\|\chi_{\varepsilon,b} J^m u(t)\|_{L^2}\\
&+\big(1+\|u\|_{L^{\infty}_TH^{\frac{1}{2}^{+}}}\big)^{M-1}\big(\|u\|_{L^{\infty}_TH^{\frac{1}{2}^{+}}}+\|\partial_x u\|_{L^{\infty}}\big) \,\mathcal{G}(u,t,\varepsilon,b) \, \|\chi_{\varepsilon,b} J^m u(t)\|_{L^2}\\
&+\big(1+\|u\|_{L^{\infty}_TH^{\frac{1}{2}^{+}}}\big)^M \|\eta_{\varepsilon,b}\, J^{m+N-1} u(t)\|_{L^2}^2\\
&+\big(1+\|u\|_{L^{\infty}_TH^{\frac{1}{2}^{+}}}\big)^{M-1}\|\partial_x u\|_{L^{\infty}} \|\chi_{\varepsilon,b}\, J^{m} u(t)\|_{L^2}^2    ,
\end{aligned}
    \end{equation*}
where
\begin{equation*}
\mathcal{G}(u,t,\varepsilon,b):=\|J^m\big(u(t) \chi_{\varepsilon,b}\big)\|_{L^2}+\|J^m\big(u(t) \phi_{\varepsilon,b}\big)\|_{L^2}+\sum_{k=2}^{M+1} \|J^m(u(t) \widetilde{\phi_{\varepsilon,b,k}})\|_{L^2}.
\end{equation*}
\end{proposition}
\begin{proof}
Given $k=1,\dots,M$, let us estimate the nonlinearity $u^k\partial_x u$. By the partition \eqref{partitionkunity}, we write
\begin{equation*}
\begin{aligned}
\chi_{\varepsilon,b} \, J^m(u^k\partial_xu) =& - [J^m,\chi_{\varepsilon,b}] u^k\partial_x u + J^m(\chi_{\varepsilon,b}\, u^k\partial_x u)\\ 
=& - \frac{1}{k+1}[J^m,\chi_{\varepsilon,b}] \partial_x (u^{k+1}) + [J^m , \chi_{\varepsilon,b}\, u^k]\partial_x u + \chi_{\varepsilon,b}\, u^k \, \partial_x J^m u\\ =& - \frac{1}{k+1}[J^m,\chi_{\varepsilon,b}] \partial_x ((u \chi_{\varepsilon,b})^{k+1} + (u \widetilde{\phi_{\varepsilon,b,k+1}})^{k+1} + u^{k+1} \psi_{\varepsilon})\\ &+ [J^m , \chi_{\varepsilon,b}\, u^k]\partial_x (u\chi_{\varepsilon,b} + u\phi_{\varepsilon,b} + u\psi_{\varepsilon}) \\ &+\chi_{\varepsilon,b}\, u^k \, \partial_x J^m u\\ =:& \sum_{\ell=1}^{7} A_5^{(1,\ell)}.
\end{aligned}
\end{equation*}
Applying Lemma \ref{CommutatorKLPV} and then Lemma \ref{Grafakos-Oh}, we have
\begin{equation*}
\begin{aligned}
\|A_5^{(1,1)}\|_{L^2}  \lesssim& \|J^l \chi_{\varepsilon,b}'\|_{L^2} \|J^{m-1} \partial_x ((u \chi_{\varepsilon,b})^{k+1})\|_{L^2} \\ 
\lesssim& \|J^m((u \chi_{\varepsilon,b})^{k+1})\|_{L^2}\\ 
\lesssim& \|u \chi_{\varepsilon,b}\|_{L^\infty}^{k} \|J^m(u \chi_{\varepsilon,b})\|_{L^2}\\ 
\lesssim&\|u \|_{L^\infty}^k \|J^m(u \chi_{\varepsilon,b})\|_{L^2}.
\end{aligned}
\end{equation*}
Similarly, we get $\|A_5^{(1,2)}\|_{L^2}\lesssim\|u \|_{L^\infty}^k \|J^m(u \widetilde{\phi_{\varepsilon,b,k+1}})\|_{L^2}$. Now, given that $\dist(\supp(\chi_{\varepsilon,b}),\supp(\psi_{\varepsilon})) \geq \varepsilon/2$, an application of Lemma \ref{IneqSupports} yields 
\begin{equation}\label{sepparatedsupp}
\begin{aligned}
\|A_5^{(1,3)}\|_{L^2} =& \frac{1}{k+1} \|\chi_{\varepsilon,b}\, \partial_x J^m(u^{k+1} \psi_{\varepsilon})\|_{L^2}\\ \lesssim& \|\chi_{\varepsilon,b}\|_{L^{\infty}} \|u^{k+1} \psi_{\varepsilon}\|_{L^2}\\
\lesssim& \|u\|_{L^{\infty}}^k \|u\|_{L^2}.
    \end{aligned}
\end{equation}
Through the Kato-Ponce inequality, Lemma \ref{Kato-Ponce}, we obtain
\begin{equation*}
\begin{aligned}
\|A_5^{(1,4)}\|_{L^2} \lesssim& \|\partial_x (\chi_{\varepsilon,b}\, u^k)\|_{L^{\infty}} \|J^m (\chi_{\varepsilon,b}\, u)\|_{L^2}+\|J^m\big(\chi_{\varepsilon,b}\, u^k\big)\|_{L^2}\|\partial_x\big(\chi_{\varepsilon,b}\, u\big)\|_{L^{\infty}}\\
\lesssim& \big(\|u\|_{L^{\infty}}^k+\|u\|_{L^{\infty}}^{k-1}\|\partial_x u\|_{L^{\infty}}\big)\|J^m (\chi_{\varepsilon,b}\, u)\|_{L^2}+\big(\| u\|_{L^{\infty}}+\|\partial_x u\|_{L^{\infty}}\big)\|J^m\big(\chi_{\varepsilon,b}\, u^k\big)\|_{L^2}. 
\end{aligned}
\end{equation*}
Note that if $k=1$, the above estimate is complete. Now, if $k=2$, we consider \eqref{partitionkunity0}, and if $k\geq 3$, we apply \eqref{partitionkunity}, together with the fact that $\chi_{\varepsilon,b}$ has separated support with $\psi_{\varepsilon}$ to get
\begin{equation*}
J^m\big(\chi_{\varepsilon,b}\, u^k\big)= J^m\big(\, (\chi_{\epsilon,b}u)^k+\chi_{\varepsilon,b}u\, (\widetilde{\phi_{\varepsilon,b,k-1}}u)^{k-1}\big),   
\end{equation*}
where, to preserve uniform notation for arbitrary $k\geq 2$, we set $\widetilde{\phi_{\varepsilon,b,k-1}}=\phi_{\varepsilon,b}$, whenever $k=2$. Next, using Lemma \ref{Grafakos-Oh}
\begin{equation}\label{estimategeneralk}
\|J^m\big(\chi_{\varepsilon,b}\, u^k\big)\|_{L^2}\lesssim \|u\|_{L^{\infty}}^{k-1}\|J^m\big(\chi_{\varepsilon,b}\, u\big)\|_{L^2}+\|u\|_{L^{\infty}}^{k-1}\|J^m\big(\widetilde{\phi_{\varepsilon,b,k-1}}\, u\big)\|_{L^2}.    
\end{equation}
We conclude that for $k\geq 1$,
\begin{equation*}
\|A_5^{(1,4)}\|_{L^2}  \lesssim \big(\|u\|_{L^{\infty}}^k+\|u\|_{L^{\infty}}^{k-1}\|\partial_x u\|_{L^{\infty}}\big)\big(\|J^m (\chi_{\varepsilon,b}\, u)\|_{L^2}+\|J^m\big(\widetilde{\phi_{\varepsilon,b,k-1}}\, u\big)\|_{L^2}\big),
\end{equation*}
where if $k=1$, we assume $\widetilde{\phi_{\varepsilon,b,k-1}}=0$. This allows us to consolidate the estimates of $\|A_5^{(1,4)}\|_{L^2}$, for any $k\geq 1$. By a similar argument above, we use Kato-Ponce inequality and \eqref{estimategeneralk} to deduce
\begin{equation*}
\begin{aligned}
\|A_5^{(1,5)}\|_{L^2} \lesssim& \big(\|u\|_{L^{\infty}}^k+\|u\|_{L^{\infty}}^{k-1}\|\partial_x u\|_{L^{\infty}}\big)(\|J^m (\chi_{\varepsilon,b}\, u)\|_{L^2}+\|J^{m}( \phi_{\varepsilon,b}\, u)\|_{L^2}+\|J^m\big(\widetilde{\phi_{\varepsilon,b,k-1}}\, u\big)\|_{L^2}).
    \end{aligned}
\end{equation*}
Using again the fact that $\chi_{\varepsilon,b}$ and $\psi_{\varepsilon}$ have separated supports, we follow a similar argument in \eqref{sepparatedsupp} to deduce
\begin{equation*}
\begin{aligned}
\|A_5^{(1,6)}\|_{L^2} \lesssim\|u\|_{L^{\infty}}^k \|u\|_{L^2}.
\end{aligned}
\end{equation*}
In order to estimate the term $A_5^{(1,7)}$, we use the integral defining $A_5$ and Lemma \ref{LocReg} (iii) as follows
\begin{equation*}
\begin{aligned}
\left|\int_{\mathbb{R}} \chi_{\varepsilon,b}\, J^m u \, A_{5}^{(1,7)} \, dx\right|  \lesssim& \frac{1}{2} \left|\int_{\mathbb{R}} (\chi_{\varepsilon,b}^2)' \, u^k \, (J^m u)^2 \, dx\right| + \frac{1}{2} \left|\int_{\mathbb{R}} \chi_{\varepsilon,b}^2 \, \partial_x\big( u^k\big) \, (J^m u)^2 \, dx\right| \\ \lesssim& \|u\|_{L^{\infty}}^k \|\sqrt{(\chi_{\varepsilon,b}^2)'} J^mu\|_{L^2}^2 + \|u\|_{L^{\infty}}^{k-1}\|\partial_x u\|_{L^{\infty}} \|\chi_{\varepsilon,b} \, J^m u\|_{L^2}^2\\ 
\lesssim& \| u \|_{L^{\infty}}^k \|\theta_{\varepsilon,b} \, J^mu\|_{L^2}^2 +\|u\|_{L^{\infty}}^{k-1} \|\partial_x u\|_{L^{\infty}} \|\chi_{\varepsilon,b} \, J^m u\|_{L^2}^2\\ 
\lesssim& \| u \|_{L^{\infty}}^{k} \|u\|_{L^2}^2 + \| u \|_{L^{\infty}}^k \|\eta_{\varepsilon,b} \, J^{m+N-1} u\|_{L^2}^2 + \|u\|_{L^{\infty}}^{k-1} \|\partial_x u\|_{L^{\infty}} \|\chi_{\varepsilon,b} \, J^m u\|_{L^2}^2.
\end{aligned}
\end{equation*}  
Hence, collecting all the previous estimates and applying Sobolev embedding, it is seen that
\begin{equation*}
\begin{aligned}
 \Big|\int_{\mathbb{R}} &\chi_{\varepsilon,b}^2 J^m u\, J^m\big(u^k\partial_x u\big)\, dx\Big|\\
 \lesssim & \|u\|_{L^{\infty}_T H^{\frac{1}{2}^{+}}}^{k+2} + \|u\|_{L^{\infty}_T H^{\frac{1}{2}^{+}}}^{k+1}\|\chi_{\varepsilon,b}\, J^m u\|_{L^2}\\
 &+\big(\|u\|_{L^{\infty}_T H^{\frac{1}{2}^{+}}}^k+\|u\|_{L^{\infty}_T H^{\frac{1}{2}^{+}}}^{k-1}\|\partial_x u\|_{L^{\infty}}\big)\Big(\|J^m(\, \chi_{\varepsilon,b}u)\|_{L^2}+\|J^m(\phi_{\varepsilon,b}u)\|_{L^2}\\
&\hspace{3cm} +\|J^m(\, \widetilde{\phi_{\varepsilon,b,k+1}}u)\|_{L^2}+\|J^m(\, \widetilde{\phi_{\varepsilon,b,k-1}}u)\|_{L^2}\Big)\|\chi_{\varepsilon,b}\, J^m u\|_{L^2}\\
& +\|u\|_{L^{\infty}_T H^{\frac{1}{2}^{+}}}^k \|\eta_{\varepsilon,b} \, J^{m+N-1} u\|_{L^2}^2+\|u\|_{L^{\infty}_T H^{\frac{1}{2}^{+}}}^{k-1}\|\partial_x u\|_{L^{\infty}}\|\chi_{\varepsilon,b}\, J^m u\|_{L^2}^2
\end{aligned}    
\end{equation*}
By summing over $k=1,\dots,M$, and using that $\|u\|_{L^{\infty}_T H^{\frac{1}{2}^{+}}}^k\leq (1+\|u\|_{L^{\infty}_T H^{\frac{1}{2}^{+}}})^k$, we get the desired result.

\end{proof}

\subsection{Proof of Theorem \ref{ThrmPropagation}}

We are now in a position to gather all the previous results to establish the propagation of regularity principle for solutions of \eqref{Benjamineq}. 

\begin{proof}[Proof of Theorem \ref{ThrmPropagation}]

As we discussed before, we shall consider $u \in C([0,T];H^{\infty}(\mathbb{R}))$ to justify the energy estimates and the validity of Propositions \ref{PropagA1}-\ref{PropagA5}. We recall this here for the sake of clarity; keeping in mind that the desired result for $u \in C([0,T];H^s(\mathbb{R}))$ solution of \eqref{Benjamineq} with initial data $u_0 \in H^s(\mathbb{R}) \cap H^m((x_0,\infty))$, $m > s > \min\{\frac{3}{2},\frac{2N+1}{4}\}$ follows from continuous dependence and by passing to the limit in our estimates. For more clarity in this regard, see the discussion at the beginning of the proof of Theorem \ref{SmoothingEffectThrm} and the ideas in \cite{IsazaLinaresPonce2015}.

We shall deduce the propagation of regularity principle by induction over $q\in \mathbb{N}$, where the regularity index $m>s$ satisfies $s+q<m\leq s+q+1$. 

\underline{Case $m = s+\widetilde{k}$ with $\widetilde{k} \in (0,1]$}. By Theorem \ref{SmoothingEffectThrm}, we know that $$\int_{0}^{T} \int_{-R}^{R} (J^{r} u)^2 \, dx \, dt < \infty$$ for any $R>0$ and $r\in [0,s+N]$. Since $$m+N-1 = s+N + (\widetilde{k}-1) \leq s+N,$$ it follows that $$\int_{0}^{T} \int_{-R}^{R} (J^{m+N-1} u)^2 \, dx \, dt < \infty$$ for any $R>0$. Let $R'$ be a positive number such that $\supp(\eta_{\varepsilon,b}(\cdot + vt)) \subseteq (-R',R')$ for all $t \in [0,T]$. Hence, by support considerations
\begin{equation} \label{SmoothBase}
\begin{aligned}
\int_{0}^{T} \|\eta_{\varepsilon,b} \, J^{m+N-1}u(t)\|_{L_x^2}^2 \, dt \leq \int_{0}^{T} \int_{-R'}^{R'} (J^{m+N-1} u)^2(x,t) \, dx \, dt < \infty.
    \end{aligned}
\end{equation}
This means that $\|\eta_{\varepsilon,b} \, J^{m+N-1}u(t)\|_{L_x^2} < \infty$ for almost every $t \in [0,T]$.

On the other hand, notice that $$J^m(u \chi_{\varepsilon,b}) = [J^m,\chi_{\varepsilon,b}] u + \chi_{\varepsilon,b} \, J^m u.$$ Since $[J^m,\chi_{\varepsilon,b}]$ is a pseudo-differential operator of order $m-1$ with $m-1 = s+(\widetilde{k}-1) \leq s$, it follows that $[J^m,\chi_{\varepsilon,b}]J^{-s}$ has order $0$. Hence, 
\begin{equation}\label{JmChiU}
\begin{aligned}
    \|J^m(u \chi_{\varepsilon,b})\|_{L^2} &\leq \|[J^m,\chi_{\varepsilon,b}]J^{-s} J^s u\|_{L^2} + \|\chi_{\varepsilon,b} \, J^m u\|_{L^2}\\ &\lesssim \|u\|_{H^s} + \|\chi_{\varepsilon,b} \, J^m u\|_{L^2}.
    \end{aligned}
\end{equation}
We recall that $\supp(\phi_{\varepsilon,b}) \subseteq [\varepsilon/4,b]$, which implies that $\theta_{\varepsilon,b} \equiv 1$ in $\supp(\phi_{\varepsilon,b})$. Hence, there exists a positive constant $\delta'$ such that $$\dist(\supp(1-\theta_{\varepsilon,b}) , \supp(\phi_{\varepsilon,b})) > \delta'.$$ Applying Lemma \ref{LocReg} (iii) and (iv), we have
\begin{equation} \label{JmPhiU}
    \begin{aligned}
        \|J^m(u \phi_{\varepsilon,b})\|_{L^2} &\lesssim \|\theta_{\varepsilon,b} J^m u\|_{L^2} +  \|u\|_{L^2} \\ &\lesssim \|\eta_{\varepsilon,b} J^{m+N-1} u\|_{L^2} +  \|u\|_{L^2}.
    \end{aligned}
\end{equation}
Given that the supports of the functions $\widetilde{\phi_{\varepsilon,b,\ell}}$ are also contained in $[\varepsilon/4,b]$, we get the same bound in \eqref{JmPhiU} for every $\|J^m(u \widetilde{\phi_{\varepsilon,b,\ell}})\|_{L^2}$.

It is worth noting that due to \eqref{JmPhiU}, in the estimation of the nonlinear part, i.e., that from Proposition \ref{PropagA5}, the main term that needs to be controlled is the following
\begin{equation*}
\begin{aligned}
\big(1+\|u\|_{L^{\infty}_T H^{\frac{1}{2}^{+}}}\big)^{M-1}&\|\partial_x u\|_{L^{\infty}} \|\eta_{\varepsilon,b} \, J^{m+N-1} u\|_{L^2}\|\chi_{\varepsilon,b} \, J^m u\|_{L^2}\\
\lesssim & \big(1+\|u\|_{L^{\infty}_T H^{\frac{1}{2}^{+}}}\big)^{2(M-1)}\|\partial_x u\|_{L^{\infty}}^2+\|\eta_{\varepsilon,b} \, J^{m+N-1} u\|_{L^2}^2\|\chi_{\varepsilon,b} \, J^m u\|_{L^2}^2,
\end{aligned}    
\end{equation*}
where we have applied Young's inequality. Note that the above expression is useful for applying Gronwall's inequality, provided that $\|\partial_x u\|_{L^2_TL_x^{\infty}}\lesssim T^{\frac{1}{4}} \|\partial_x u\|_{L^4_TL_x^{\infty}}$, and that $\|\eta_{\varepsilon,b} \, J^{m+N-1} u\|_{L^2_TL^2_x}<\infty$.

Gathering Propositions \ref{PropagA1}-\ref{PropagA5}, the inequalities \eqref{SmoothBase}, \eqref{JmChiU}, and \eqref{JmPhiU} into the identity \eqref{EqProp2}, we arrive at a differential inequality which, after applying Gronwall's inequality, gives us
\begin{equation*}
    \sup_{t \in [0,T]} \int_{\mathbb{R}} \chi_{\varepsilon,b}^2 \,  (J^m u)^2 \, dx + \left( C_N -\omega \right) \int_0^T \int_{\mathbb{R}} (\chi_{\varepsilon,b}^2)' \, (J^{m+N} u)^2 \, dx \, dt \leq C_0,
\end{equation*}
where 
\begin{equation*}
C_0 = C_0\left(\|J^m u_0\|_{L^2((x_0,\infty))};\|\eta_{\varepsilon,b}J^{m+N-1}u\|_{L_T^2 L_x^2};\|\partial_x u\|_{L_T^2 L_x^{\infty}};\|u\|_{L_T^{\infty} H^s_x};\varepsilon;b;v;T\right)>0    
\end{equation*}
for any $\varepsilon>0$, $b\geq 5\varepsilon$, $v>0$, $0<\omega<C_N$ (this constant is given by Proposition \ref{PropagA3}). As we mentioned earlier, the above inequality shows that an approximation argument by smooth solutions of \eqref{Benjamineq} provides the desired propagation of regularity principle for the present case.
\\ \\
\underline{Case $m=s+q+\widetilde{k}$ with $q\in\mathbb{Z}^+$ and $\widetilde{k} \in (0,1]$}. Let us assume by inductive hypothesis, that 
\begin{equation*}
    \begin{aligned}
        \sup_{t\in [0,T]} \int_{\mathbb{R}}& \chi_{\varepsilon',b'}^2 \,  (J^{\overline{m}} u)^2 \, dx + \int_0^T \int_{\mathbb{R}} (\chi_{\varepsilon',b'}^2)' \, (J^{\overline{m}+N} u)^2 \, dx \, dt \leq C
    \end{aligned}
\end{equation*}
for every $\overline{m} \in (s, s+q]$ and for any $\varepsilon'>0$, $b'\geq 5\varepsilon'$, and $v'>0$. Our goal is to show that \eqref{propagation} and \eqref{smoothing} also hold for $m=s+q+\widetilde{k}$. 

Let $\varepsilon>0$, $b\geq 5\varepsilon$, and $v>0$ be arbitrary but fixed. Taking $\varepsilon' = \varepsilon/24$, $b' = b+13\varepsilon/12$, and $\overline{m} = m-1 \leq s+q$, the induction hypothesis implies $$\int_0^T \int_{\mathbb{R}} (\chi_{\varepsilon/24,b+13\varepsilon/12}^2)' \, (J^{m-1+N} u)^2 \, dx \, dt < \infty.$$ From the construction of the family $\{\chi_{\varepsilon,b}\}$, we know that there exists a positive constant $C_{\varepsilon,b}$ such that $$(\chi_{\varepsilon/24,b+13\varepsilon/12})  (\chi_{\varepsilon/24,b+13\varepsilon/12}') \geq C_{\varepsilon,b} \mathds{1}_{[\varepsilon/8,b+\varepsilon]}.$$ Thus, we find
\begin{equation} \label{smoothingInductive}
    \begin{aligned}
        \int_{0}^{T} \|\eta_{\varepsilon,b} \, J^{m+N-1}u(t)\|_{L_x^2}^2 \, dt  &\leq \int_{0}^{T} \int_{\mathbb{R}} \mathds{1}_{[\varepsilon/8,b+\varepsilon]}(x+vt) \, (J^{m+N-1} u)^2(x,t) \, dx \, dt \\ &\leq \frac{1}{2C_{\varepsilon,b}}\int_{0}^{T} \int_{\mathbb{R}} (\chi_{\varepsilon',b'}^2)'(x+vt) \, (J^{m+N-1} u)^2(x,t) \, dx \, dt\\ &< \infty.
    \end{aligned}
\end{equation}
It follows that $\|\eta_{\varepsilon,b} \, J^{m+N-1}u(t) \|_{L^2_x}<\infty$ for almost every $t\in [0,T]$. Notice that the terms $\|J^m(u \phi_{\varepsilon,b})\|_{L^2}$ and each $\|J^m(u \widetilde{\phi_{\varepsilon,b,\ell}})\|_{L^2}$ can be controlled using Lemma \ref{LocReg} as we did in \eqref{JmPhiU}.

Now, we shall estimate the $L^2$-norm of the term $J^m(u \chi_{\varepsilon,b})$. Let $m_1$ be a positive integer such that $m_1 \geq m-s-1$. By Lemma \ref{LemmaAllpseudo}, there is a pseudo-differential operator $\Psi_{m-m_1-1}\in  \mathrm{OP}\mathcal{S}^{m-m_1-1}$, and for each $j\in\{1,\dots,m_1\}$ there exist a constant $a_j$ and a pseudo-differential operator $\Psi^j  \in \mathrm{OP}\mathcal{S}^{0}$ such that 
\begin{equation*}
    \begin{aligned}
        J^m (u \, \chi_{\varepsilon,b}) &= [J^m , \chi_{\varepsilon,b}] u + \chi_{\varepsilon,b} \, J^m u\\ &= \sum_{j=1}^{m_1} a_j \, \chi_{\varepsilon,b}^{(j)} \, \Psi^j J^{m-j} u + \Psi_{m-m_1-1} u + \chi_{\varepsilon,b} \, J^m u.
    \end{aligned}
\end{equation*}
Since $m-m_1-1\leq s$, the pseudo-differential operator $\Psi_{m-m_1-1} J^{-s}$ determines a bounded operator in $L^2(\mathbb{R})$. Thus, an application of Lemma \ref{LocReg} (iii) yields
\begin{equation} \label{JmChiUPart2}
    \begin{aligned}
        \|J^m (u \, \chi_{\varepsilon,b})\|_{L^2} \lesssim& \sum_{j=1}^{m_1}  \|\chi_{\varepsilon,b}^{(j)} \, \Psi^j J^{m-j} u\|_{L^2} + \|\Psi_{m-m_1-1}J^{-s}(J^s u)\|_{L^2} + \|\chi_{\varepsilon,b} \, J^m u\|_{L^2} \\  \lesssim& \sum_{j=1}^{m_1}  \left(\|\chi_{\varepsilon,b}^{(j)} \, \Psi^j( \theta_{\varepsilon,b} \, J^{m-j} u)\|_{L^2} +\|\chi_{\varepsilon,b}^{(j)} \, \Psi^j( (1-\theta_{\varepsilon,b}) \, J^{m-j} u)\|_{L^2}  \right) \\ &+ \|J^s u\|_{L^2} + \|\chi_{\varepsilon,b} \, J^m u\|_{L^2}\\  \lesssim& \sum_{j=1}^{m_1}  \left(\|\theta_{\varepsilon,b} \, J^{m-j} u\|_{L^2} +\|\chi_{\varepsilon,b}^{(j)} \, \Psi^j( (1-\theta_{\varepsilon,b}) \, J^{m-j} u)\|_{L^2}  \right) \\ &+ \|J^s u\|_{L^2} + \|\chi_{\varepsilon,b} \, J^m u\|_{L^2}\\
        \lesssim &\sum_{j=1}^{m_1}  \left(\|\eta_{\varepsilon,b} \, J^{m+N-1} u\|_{L^2} + \|u\|_{L^2} \right) + \|J^s u\|_{L^2} + \|\chi_{\varepsilon,b} \, J^m u\|_{L^2},
    \end{aligned}
\end{equation}
where given that the distance between $\supp(\chi_{\varepsilon,b}^{(j)})$ and $\supp((1-\theta_{\varepsilon,b}))$ is greater than a positive constant, we have used the separated support property to deduce $\|\chi_{\varepsilon,b}^{(j)} \, \Psi^j( (1-\theta_{\varepsilon,b}) \, J^{m-j} u)\|_{L^2}\lesssim \|u\|_{L^2}$.

Gathering the previous results, by \eqref{smoothingInductive}-\eqref{JmChiUPart2}, we apply Proposition \ref{PropagA1}-\ref{PropagA5} and Gronwall's inequality to deduce
\begin{equation*}
    \sup_{t \in [0,T]} \int_{\mathbb{R}} \chi_{\varepsilon,b}^2 \,  (J^m u)^2 \, dx + \left( C_N -\omega \right) \int_0^T \int_{\mathbb{R}} (\chi_{\varepsilon,b}^2)' \, (J^{m+N} u)^2 \, dx \, dt \leq C_0,
\end{equation*}
where 
\begin{equation*}
C_0 = C_0\left(\|J^m u_0\|_{L^2((x_0,\infty))};\|\eta_{\varepsilon,b}J^{m+N-1}u\|_{L_T^2 L_x^2};\|\partial_x u\|_{L_T^2 L_x^{\infty}};\|u\|_{L_T^{\infty} H^s_x};\varepsilon;b;v;T\right)>0.    
\end{equation*}
Since $m=s+q+\widetilde{k}$ with $\widetilde{k}\in (0,1]$, the above estimate completes the deduction of the inductive step and, in turn, finishes the proof of the theorem. 
\end{proof}


\section{Appendix: proof of Lemma \ref{LWPN12}}\label{LWPresults}

In this part, we prove that the Cauchy problem \eqref{Benjamineq} is LWP in $H^s$-spaces. We focus on the proof of Lemma \ref{LWPN12} as the proof of Lemma \ref{StandardLWP} follows by standard techniques. We adapt the strategy in \cite{KenigPonceVega1991,Laurey1997} to the Cauchy problem \eqref{Benjamineq}.

We consider the linear Cauchy problem
\begin{equation}\label{linearBE}
\left\{\begin{aligned}
&\partial_t u+\gamma\mathcal{H}\partial_x^2 u+(-1)^{N+1} \partial_x^{2N+1}u+ P(D) u=0, \quad x\in \mathbb{R},\, t\in \mathbb{R},\\
&u(x,0)=u_0(x).
\end{aligned}\right. 
\end{equation}
The equation in \eqref{linearBE} generates the unitary group $\{S(t)\}_{t\in \mathbb{R}}$ in $H^s$-spaces, which is defined via the Fourier transform by 
\begin{equation}\label{LinearGdef}
    S(t)f=(e^{it\omega(\xi)}\widehat{f}(\xi))^{\vee},
\end{equation}
where
\begin{equation}\label{dispersiondef}
    \omega(\xi):=-\gamma \xi |\xi|+\xi^{2N+1}-\sum_{k=1}^{N-1} (-1)^k  a_k \xi^{2k+1},
\end{equation}
we follow the zero convention for the empty sum $\sum_{k=1}^{N-1}(\cdots)=0$, when $N=1$. Thus, the integral formulation of \eqref{Benjamineq} is given by
\begin{equation}\label{inteeq}
u(t)=S(t)u_0-\int_0^t S(t-\tau)\left(\sum_{k=1}^M b_k u^k \partial_x u\right)(\tau)\, d\tau.
\end{equation}

\subsection{Linear estimates} 

We recall the following Strichartz estimates for the group $\{S(t)\}_{t\in\mathbb{R}}$. Their proof can be consulted in \cite{Laurey1997}, see also \cite{KenigPonceVega1991,KenigPonceVega1991I}.

\begin{lemma}\label{StrEstimates}
Let $N\in \mathbb{Z}^{+}$, $\gamma\in \mathbb{R}$. If $N\geq 2$, consider $a_1,\dots,a_{N-1}\in \mathbb{R}$. Let $S(t)$ be defined by \eqref{LinearGdef} with $\omega(\xi)$ as in \eqref{dispersiondef}. Then, for any $T>0$ 
    \begin{equation*}
        \||D|^N S(t)u_0\|_{ L^{\infty}_xL^2_T}\lesssim \langle T\rangle^{1/2}\|u_0\|_{L^2} \qquad \text{(Kato's smoothing effect)}.
    \end{equation*}
For any $\theta \in [0,1]$, $(q,p)=\big(\frac{4}{\theta},\frac{2}{(1-\theta)}\big)$, $\frac{1}{p}+\frac{1}{p'}=\frac{1}{q}+\frac{1}{q'}=1$, and $T>0$,
\begin{equation*}
\||D|^{\frac{\theta(2N-1)}{4}} S(t)u_0\|_{L^q_T L^p_x}\lesssim \langle T\rangle^{\theta/4}\|u_0\|_{L^2} \quad \text{(Strichartz's estimates)}  
\end{equation*}
and its dual version
\begin{equation*}
\Big\|\int_0^t |D|^{\frac{\theta(2N-1)}{2}} S(t-\tau)f(\tau)\, d\tau\Big\|_{L^q_T L^p_x}\leq \langle T\rangle^{\theta/2}\|f\|_{L^{q'}_TL^{p'}_x} \quad \text{(Strichartz's estimates)}.   
\end{equation*}
Let $s>\frac{2N+1}{4}$. For any $T>0$,
\begin{equation*}
\|S(t)u_0\|_{L^2_xL^{\infty}_T}  \lesssim \langle T \rangle^{1/2}\|u_0\|_{H^s} \quad \text{(Maximal function)}.
\end{equation*}
    
\end{lemma}

\begin{remark}\label{remkKato}
We can obtain the following version of Kato's smoothing effect in Lemma \ref{StrEstimates}
\begin{equation*}
\|J^N S(t)u_0\|_{ L^{\infty}_xL^2_T}\lesssim \langle T\rangle^{1/2}\|u_0\|_{L^2}.
\end{equation*}
To observe this, let $\psi \in C^{\infty}_0(\mathbb{R})$ be such that $0\leq \psi\leq 1$, and $\psi(\xi)=1$ for all $|\xi|\leq 1$. We denote by $P_{\leq 1}$ the Fourier multiplier by the function $\psi$. Let $P_{>1}=I-P_{\leq 1}$. Thus, using Sobolev embedding, the properties of the operators $P_{\leq 1}$ and $P_{>1}$, and Lemma \ref{StrEstimates}, we arrive at
\begin{equation*}
 \begin{aligned}
\|J^N S(t)u_0\|_{ L^{\infty}_xL^2_T}\lesssim & \|J^N P_{\leq 1} S(t)u_0\|_{ L^{\infty}_xL^2_T}+\|J^N P_{> 1} S(t)u_0\|_{ L^{\infty}_xL^2_T}\\
\lesssim & \langle T\rangle^{1/2}\|J^{N+\frac{1}{2}^{+}}P_{\leq 1} S(t)u_0\|_{ L^{\infty}_T L^2_x}+\| |D|^{N} S(t)(|D|^{-N} J^N P_{> 1} u_0)\|_{ L^{\infty}_xL^2_T}\\
\lesssim & \langle T\rangle^{1/2}\big(\|u_0\|_{L^2}+\||D|^{-N} J^N P_{> 1} u_0)\|_{L^2}\big).
 \end{aligned}   
\end{equation*}   
Using that $|D|^{-N} J^N P_{> 1}$ is a bounded operator in $L^2(\mathbb{R})$, we get the desired estimate.

Under the assumptions over $(p,q)$ in Lemma \ref{StrEstimates}, we can apply a similar argument as above, together with the embeddings $H^{\frac{\theta}{2}}(\mathbb{R})\hookrightarrow L^{p}(\mathbb{R})$ for $\theta\in(0,1)$ and $H^{\frac{1}{2}^{+}}(\mathbb{R})\hookrightarrow L^{\infty}(\mathbb{R})$  to deduce
\begin{equation*}
\|J^{\frac{\theta(2N-1)}{4}} S(t)u_0\|_{L^q_T L^p_x}\lesssim \langle T\rangle^{\theta/4}\|u_0\|_{L^2}.  
\end{equation*}

\end{remark}

We are now in a position to deduce our LWP.

\begin{proof}[Proof of Lemma \ref{LWPN12}]

Given $s>\frac{2N+1}{4}$, let $u_0\in H^s(\mathbb{R})$ arbitrary but fixed.  We define the space
\begin{equation}\label{Xspacedef}
X(T,a)=\{u\in C([0,T];H^s(\mathbb{R})): \, \vertiii{u}\leq a\},    
\end{equation}
where
\begin{equation}\label{tripnormdef}
\begin{aligned}
\vertiii{u}:=\|J^s u\|_{L^{\infty}_TL^2_x}+\||D|^s \partial_x u\|_{L^{\infty}_xL^2_T} +\|J^{s-\frac{(5-2N)}{4}}\partial_x u\|_{L^4_T L^{\infty}_x}+\|u\|_{L^2_x L^{\infty}_T}.
\end{aligned}   
\end{equation}
Notice that for any $a,T>0$, $X(T,a)$ defines a complete metric space. Motivated by the integral formulation of \eqref{Benjamineq}, we consider the function
\begin{equation*}
\Phi(u)=S(t)u_0-\int_0^t S(t-\tau)\left(\sum_{k=1}^M b_k u^k \partial_x u\right)(\tau)\, d\tau, \quad u\in X(T,a).
\end{equation*}
In what follows, we shall prove that there exist a real number $a>0$ and a time $T>0$ such that $\Phi: X(T,a) \rightarrow X(T,a)$ is well-defined. 
\\ \\
\underline{\bf $\|\cdot\|_{L^{\infty}_TL^2_x}$-norm estimates}. Since each operator in the group $\{S(t)\}_{t\in \mathbb{R}}$ defines an isometry in $L^2(\mathbb{R})$, we get
\begin{equation*}
\begin{aligned}
 \|J^s\Phi(u)\|_{L^{\infty}_T L^2_x}
 \leq & \|S(t) J^s u_0\|_{L_T^{\infty} L_x^2}+c\sum_{k=1}^M\Big\|\int_0^t S(t-\tau)J^s(u^k\partial_x u)(\tau)\, d\tau\Big\|_{L^{\infty}_T L_x^2}\\
  \leq & \|J^s u_0\|_{L^2}+c\sum_{k=1}^M \int_0^T \|J^s(u^k\partial_x u)(\tau)\|_{L_x^2}\, d\tau\\
 = & \|J^s u_0\|_{L^2} +c\sum_{k=1}^M\|J^s(u^k \partial_x u)\|_{L^1_T L^2_x}, 
\end{aligned}
\end{equation*}
where $c$ depends on the $b_k$'s. For every $k\in \{1,\dots,M\}$, we have
\begin{equation*}
\begin{aligned}
\|J^s(u^k \partial_x u)\|_{L^1_T L^2_x}\leq & \|[J^s,u^k]\partial_x u\|_{L^1_T L^2_x}+\|u^k (J^s-|D|^s)\partial_x u\|_{L^1_T L^2_x}\\
&+\|u^k |D|^s\partial_x  u\|_{L^1_T L^2_x}.       
\end{aligned}    
\end{equation*}
We estimate each term on the right-hand side of the above inequality. But first, we recall that the Bessel potential $J^{-s'}$ can be written as $J^{-s'} h = G_{-s'} \ast h$ for any $s'>0$, $p\in [1,\infty]$, and $h\in L^p(\mathbb{R}^d)$, where the kernel $G_{-s'}$ satisfies $\|G_{-s'}\|_{L^1} = 1$ (see \cite[Chapter V]{Stein1970Singular}). By Young's inequality,
\begin{equation} \label{BessIneq}
\|J^{-s'} h\|_{L^p}=\|G_{-s'}\ast h\|_{L^p} \leq \|G_{-s'}\|_{L^1} \|h\|_{L^p} = \|h\|_{L^p}
\end{equation}
for any $p\in [1,\infty]$ and $s'\geq 0$. By \eqref{BessIneq}, we know that $\|J^{-s+\frac{5-2N}{4}}f\|_{L^p}\leq \|f\|_{L^p}$ for all $1\leq p \leq \infty$, which implies
\begin{equation}\label{dxcontrol}
\|\partial_x u\|_{L^4_T L^{\infty}_x}=\|J^{-s+\frac{5-2N}{4}}J^{s-\frac{(5-2N)}{4}}\partial_x u\|_{L^4_T L^{\infty}_x}\leq \|J^{s-\frac{(5-2N)}{4}}\partial_x u\|_{L^4_T L^{\infty}_x}.    
\end{equation}
Using Lemma \ref{Kato-Ponce}, the fact that $H^s(\mathbb{R})$ is a Banach algebra provided that $s>\frac{1}{2}$, and H\"older's inequality, it is seen that
\begin{equation*}
\begin{aligned}
\|[J^s,u^k]\partial_x u\|_{L^1_T L^2_x}\lesssim & \|\|\partial_x(u^k)\|_{L^{\infty}_x}\|J^s u\|_{L^{2}_x}\|_{L^1_T}+\|\|J^s(u^k)\|_{L^2_x}\|\partial_x u\|_{L^{\infty}}\|_{L^1_T} \\
\lesssim &  \|J^s u\|_{L^{\infty}_TL^2_x}^k\|\partial_x u\|_{L^1_T L^{\infty}_x}\\
\lesssim & T^{3/4} \|J^s u\|_{L^{\infty}_T L^2_x}^k\|J^{s-\frac{(5-2N)}{4}}\partial_x u\|_{L^4_T L^{\infty}_x}\\
\lesssim & T^{3/4}\vertiii{u}^{k+1},
\end{aligned}    
\end{equation*}
where we used Sobolev embedding $H^s(\mathbb{R})\hookrightarrow L^{\infty}(\mathbb{R})$ (i.e., $\|f\|_{L^{\infty}}\lesssim \|f\|_{H^s}$). We remark that the argument employed in \eqref{dxcontrol} also shows that 
\begin{equation*}
\|\partial_ x J^r u\|_{L^4_T L^{\infty}_x} \leq \|J^{s-\frac{(5-2N)}{4}}\partial_x u\|_{L^4_T L^{\infty}_x},
\end{equation*} 
for all $0\leq r \leq s-\frac{(5-2N)}{4}$. This justifies why the solution derived from our arguments satisfies \eqref{solclass2}.  

Next, by Proposition \ref{ExpDs-Js} and Sobolev embedding, it follows 
\begin{equation*}
\begin{aligned}
 \|u^k (J^s-|D|^s)\partial_x u\|_{L^1_T L^2_x}\lesssim& \|\|u\|_{L^{\infty}_x}^k\|J^s u\|_{L^2_x}\|_{L^1_T}\lesssim & T\|J^su\|_{L^{\infty}_T L^2_x}^{k+1}\\
\lesssim & T \vertiii{u}^{k+1}.
\end{aligned}    
\end{equation*}
By H\"older's inequality, using that $\|\cdot\|_{L^p_T L^p_x}=\|\cdot\|_{L^p_x L^p_T}$ for any $p\in [1,\infty]$ and Sobolev embedding, we deduce
\begin{equation*}
\begin{aligned}
\|u^k |D|^s\partial_x  u\|_{L^1_T L^2_x}\lesssim &  T^{1/2}\|u^k |D|^s\partial_x  u\|_{L^2_x L^2_T}\\
\lesssim & T^{1/2}\|u^k\|_{L^2_x L^{\infty}_T} \||D|^s\partial_x  u\|_{L^{\infty}_x L^2_T}\\
\lesssim & T^{1/2}\|J^su\|_{L^{\infty}_TL^2_x}^{k-1}\|u\|_{L^2_x L^{\infty}_T} \||D|^s\partial_x  u\|_{L^{\infty}_x L^2_T}\\
\lesssim & T^{1/2}\vertiii{u}^{k+1}.
\end{aligned}    
\end{equation*}
Summarizing the previous results, there exists a constant $c>0$ such that
\begin{equation*}
\begin{aligned}
 \|J^s\Phi(u)\|_{L^{\infty}_T L^2_x} \lesssim & \|u_0\|_{H^s} +c(T^{1/2}+T^{3/4}+T)\sum_{k=1}^M  \vertiii{u}^{k+1}.
\end{aligned}
\end{equation*}
\\ \\
\underline{\bf $\||D|^s \partial_x\big(\cdot\big)\|_{L^{\infty}_x L^2_T}$-norm estimates}. We apply Kato's smoothing effect (Lemma \ref{StrEstimates} and Remark \ref{remkKato}), along with earlier arguments used to estimate $\|J^s(u^k\partial_x u)\|_{L^1_T L^2_x}$ to deduce
\begin{equation*}
\begin{aligned}
\||D|^s\partial_x \Phi(u)\|_{L^{\infty}_xL^2_T}\lesssim & \|J^N|D|^{s}J^{-N}\partial_x S(t) u_0\|_{L^{\infty}_xL^2_T}\\
&+ \sum_{k=1}^M \Big\|J^N|D|^{s}J^{-N}\partial_x\int_0^t S(t-\tau) (u^k \partial_x u)(\tau)\, d\tau\Big\|_{L^{\infty}_xL^2_T}\\
\lesssim & \langle T\rangle^{1/2}\||D|^{s}J^{-N}\partial_x u_0\|_{L^2}\\
&+\langle T\rangle^{1/2}\sum_{k=1}^M\int_0^{T} \||D|^{s}J^{-N}\partial_x\big(u^k \partial_x u\big)(\tau)\|_{L^2_x}\, d\tau\\
\lesssim & \langle T\rangle^{1/2}\|J^su_0\|_{L^2}+\langle T\rangle^{1/2}\sum_{k=1}^M\|J^s\big(u^k \partial_x u\big)\|_{L^{1}_T L^2_x}\\
\lesssim & \langle T\rangle^{1/2}\|J^su_0\|_{L^2}+\langle T\rangle^{1/2} (T^{1/2}+T^{3/4}+T)\sum_{k=1}^M  \vertiii{u}^{k+1}.
\end{aligned}    
\end{equation*}
\\ \\
\underline{\bf $\|J^{s-\frac{(5-2N)}{4}} \partial_x\big(\cdot\big)\|_{L^{4}_T L^{\infty}_x}$-norm estimates}. We apply the Strichartz estimate in Lemma \ref{StrEstimates} and its version with $J$ in Remark \ref{remkKato} with $\theta=1$ to get 
\begin{equation*}
\begin{aligned}
\|J^{s-\frac{(5-2N)}{4}}&\partial_x\Phi(u)\|_{L^4_T L^{\infty}_x}\\
\lesssim &  \|J^{\frac{2N-1}{4}}S(t)J^{-\frac{(2N-1)}{4}}\partial_x J^{s-\frac{(5-2N)}{4}}u_0\|_{L^4_T L^{\infty}_x} \\
&+\sum_{k=1}^M\Big\|J^{\frac{2N-1}{4}}\int_0^t S(t-\tau) J^{-\frac{(2N-1)}{4}}\partial_x J^{s-\frac{(5-2N)}{4}}(u^k\partial_x u)(\tau)\, d\tau\Big\|_{L^4_T L^{\infty}_x}\\
\lesssim & \langle T\rangle^{1/4}\|J^su_0\|_{L^2}+\sum_{k=1}^{M}\langle T\rangle^{1/4}\|J^s\big(u^k \partial_x u\big)\|_{L^1_T L^2_x}\\
\lesssim & \langle T\rangle^{1/4}\|J^su_0\|_{L^2}+\langle T\rangle^{1/4} (T^{1/2}+T^{3/4}+T)\sum_{k=1}^M  \vertiii{u}^{k+1}.
\end{aligned}    
\end{equation*}
\\ \\
\underline{\bf $\|\cdot\|_{L^{2}_x L^{\infty}_T}$-norm estimates}. Using the maximal function estimate in Lemma \ref{StrEstimates} and the fact that $s>\frac{2N+1}{4}$, we find
\begin{equation*}
\begin{aligned}
\|\Phi(u)\|_{L^2_x L^{\infty}_T} \leq &\|S(t)u_0\|_{L^2_xL^{\infty}_T} +\sum_{k=1}^M\Big\|\int_0^t S(t-\tau)(u^k\partial_x u)(\tau)\, d\tau\Big\|_{L^2_xL^{\infty}_T}\\
\lesssim & \langle T\rangle^{1/2}\|J^s u_0\|_{L^2}+\langle T\rangle^{1/2}\sum_{k=1}^M\|J^s(u^k\partial_x u)\|_{L^1_T L^2_x}\\
\lesssim & \langle T\rangle^{1/2}\|J^su_0\|_{L^2}+\langle T\rangle^{1/2} (T^{1/2}+T^{3/4}+T)\sum_{k=1}^M  \vertiii{u}^{k+1}.
\end{aligned}    
\end{equation*}
Gathering the previous results, there exists $c^{\ast}>0$ such that
\begin{equation}\label{contrc1}
\begin{aligned}
\vertiii{\Phi(u)}\leq c^{\ast}\langle T\rangle^{1/2}
\|J^s u_0\|_{L^2}+c^{\ast}T^{1/2}\langle T\rangle^{1/2} \sum_{k=1}^M  \vertiii{u}^{k+1}. 
\end{aligned}    
\end{equation}
Setting $a=2c^{\ast}\|J^s u_0\|_{L^2}$, we consider a time $T>0$ small enough such that
\begin{equation*}
\begin{aligned}
&\langle T\rangle^{1/2}\frac{a}{2}+c^{\ast}T^{1/2}\langle T \rangle \sum_{k=1}^M a^{k+1}\leq a.
\end{aligned}
\end{equation*}
Thus, from this choice of $a$ and $T$, \eqref{contrc1}, and a simple continuity argument, show that $\Phi:X(T,a)\rightarrow X(T,a)$ is well-defined. Moreover, similar arguments show that $\Phi$ is a contraction. The rest of the proof employs standard arguments; therefore, we will omit further details.
\end{proof}


\section*{Acknowledgments}

The authors express their gratitude to the seminar \emph{“Semillero de An\'alisis Arm\'onico y EDP”} at the Universidad Nacional de Colombia, Bogot\'a D.C., where several parts of this work were presented and discussed. C. G. acknowledges the financial support provided by the \emph{“Grado de Honor”} scholarship from the same university. We are grateful to Professor Gustavo Ponce for bringing to our attention the validity of Corollary \ref{corollarypropdecay}.


\bigskip

{\bf Data Availability.} 
Data sharing not applicable to this article as no datasets were generated or analyzed during the current study.

\bigskip

{\bf Competing Interests.} 
The authors have no competing interests to declare that are relevant to the content of this article.


\bibliographystyle{acm}

\bibliography{references1}

\end{document}